\documentclass[10pt]{smfart}

\RequirePackage[T1]{fontenc}
\RequirePackage{amsfonts,latexsym,amssymb}
\RequirePackage[frenchb]{babel}
\addto\extrasfrenchb{\bbl@nonfrenchitemize
\bbl@nonfrenchspacing}

\RequirePackage{mathrsfs}
\let\mathcal\mathscr

\usepackage{array}
\usepackage{float}
\usepackage{color}
\usepackage[all]{xy}
\usepackage{url}
\usepackage{fge}
\newcommand{\moins}{\mathbin{\fgebackslash}}

\numberwithin{equation}{section}
\theoremstyle{plain}
\newtheorem{prop}[equation]{\propname}
\newtheorem{theo}[equation]{\theoname}

\newtheorem{coro}[equation]{\coroname}

\newtheorem{lemm}[equation]{\lemmname}
\theoremstyle{definition}
\theoremstyle{remark}

\newtheorem{rema}[equation]{\remaname}
\newtheorem{exem}[equation]{\exemname}

\let\cal\mathcal

\let\bb\mathbb

\def\cdo{{\rm Mes}}

\def\cZ{{\widehat\Z}}

\def\point{{\cdot}}

\def\Q{{\bf Q}} \def\Z{{\bf Z}}
\def\C{{\bf C}}
\def\N{{\bf N}}

\def\O{{\cal O}}
\def\dual{{\boldsymbol *}}
\def\bmu{{\boldsymbol\mu}}

\def\epsilon{\varepsilon}

\def\oe{\O_{\cal E}}
\def\woe{\widetilde\O_{\cal E}}
\def\we{\widetilde{\cal E}}

\def\ainf{{\bf A}_{{\rm inf}}}

\def\E{{\bf E}} \def\A{{\bf A}} \def\B{{\bf B}}
\def\tE{\widetilde{\bf E}} \def\tA{\widetilde{\bf A}} \def\tB{\widetilde{\bf B}}

\def\BB{\mathbb B} \def\PP{\mathbb P} \def\GG{\mathbb G} \def\UU{\mathbb U}

\def\piqp{{{\bf P}^1}}

\def\reso{{\text{r\'es}_0}}

\def\matrice#1#2#3#4{{\big(\begin{smallmatrix}#1&#2\\ #3&#4\end{smallmatrix}\big)}}

\def\wotimes{\widehat{\otimes}}

\def\GG{{\mathbb G}} \def\BB{{\mathbb B}} \def\PP{{\mathbb P}} \def\UU{{\mathbb U}}
\def\cZ{\widehat{\Z}}

\makeindex
\begin{document}
\title[Fonctions d'une variable $p$-adique]
{Fonctions d'une variable $p$-adique et repr\'esentations de ${\rm GL}_2(\Q_p)$}
\author{Pierre Colmez}
\address{C.N.R.S., IMJ-PRG, Sorbonne Universit\'e, 4 place Jussieu,
75005 Paris, France}
\email{pierre.colmez@imj-prg.fr}
\author{Shanwen Wang}
\address{School of mathematics, Renmin University of China, 59 ZhongGuanCun Street, 100872 Beijing, P.R. China}
\email{s\_wang@ruc.edu.cn}

\begin{abstract}
Nous \'etendons le dictionnaire entre anneaux de Fontaine et analyse fonctionnelle $p$-adique,
et nous donnons un raffinement de la correspondance de Langlands locale $p$-adique pour les
repr\'esentations de la s\'erie principale de ${\rm GL}_2(\Q_p)$. 
\end{abstract}
\begin{altabstract}
We extend the dictionary between Fontaine rings and $p$-adic functionnal analysis, and we give a refinement
of the $p$-adic local Langlands correspondence for principal series representations of ${\rm GL}_2(\Q_p)$.
\end{altabstract}
\thanks{Pendant l'\'elaboration de cet article la recherche de S.W. a \'et\'e subventionn\'ee par
the Fundamental Research Funds for the Central Universities, and the Research Funds of Renmin University 
of China \no 20XNLG04 et the National Natural Science Foundation of China (Grant \no 11971035); 
P.C. \'etait membre du projet ANR Coloss.}
\setcounter{tocdepth}{1}

\maketitle
{\Small
\tableofcontents}

\section*{Introduction}
 On fixe une cl\^oture alg\'ebrique $\overline{\Q}_p$ de $\Q_p$, 
et on note $G_{\Q_p}$ le groupe de Galois absolu de $\Q_p$
et $H_{\Q_p}\subset G_{\Q_p}$ celui de $\Q_p(\bmu_{p^{\infty}})$. 
 Soit $\C_p$ le compl\'et\'e de $\overline{\Q}_p$ pour la valuation $p$-adique $v_p$,
et soient 
$\widetilde{\bf A}=W(\C_p^\flat)$, $\widetilde{\bf A}^+=W(\O_{\C_p^{\flat}})$ 
et $\widetilde{\bf A}^{-}= \widetilde{\bf A}/\widetilde{\bf A}^+$. 

Notons $\GG$ le groupe ${\bf GL}_2(\Q_p)$, $\BB$ son sous-groupe de Borel sup\'erieur 
et ${\bb P}$ son sous-groupe mirabolique. 

Soit $L$ une extension finie de $\Q_p$ et soit $\O_L$ l'anneau de ses entiers.
Si $\Pi$ est une $L$-repr\'esentation unitaire de $\GG$, de longueur finie, et si ${\bf V}(\Pi)$
est la repr\'esentation de $G_{\Q_p}$ qui lui est associ\'ee par la correspondance de Langlands locale
$p$-adique, il est prouv\'e dans~\cite{CW} que l'on dispose 
d'une injection naturelle 
$\Pi\hookrightarrow (\tA^-\otimes {\bf V}(\Pi))^{H_{\Q_p}}$, 
$\PP$-\'equivariante.
Ceci joue un grand r\^ole pour
construire un mod\`ele de Kirillov ``en famille'' que l'on peut comparer au mod\`ele de Kirillov de
la cohomologie compl\'et\'ee de la tour des courbes modulaires
et obtenir de la sorte une nouvelle preuve de la factorisation d'Emerton.

Maintenant, $\Pi$ contient le sous-$\PP$-module 
$(\tA\otimes {\bf V}(\Pi))^{H_{\Q_p}}/(\tA^+\otimes {\bf V}(\Pi))^{H_{\Q_p}}$ et le
quotient $J(\Pi)$ est de dimension finie (c'est le module de Jacquet de $\Pi$).
Ce quotient s'injecte dans $H^1(H_{\Q_p},\tA^+\otimes{\bf V}(\Pi))$ qui est un espace
de dimension infinie, et la question se pose de d\'ecrire l'image: contient-elle
de l'information sur $\Pi$ qui n'est pas encod\'ee par ${\bf V}(\Pi)$?
Le but de ce petit article est de montrer que "oui".
Notons que, si ${\bf V}(\Pi)$ est irr\'eductible,
 la question ne se pose que pour les repr\'esentations de la s\'erie
principale car sinon $J(\Pi)=0$ et on peut reconstruire $\Pi$ \`a partir de ${\bf V}(\Pi)$.

Dans le cas d'une s\'erie principale, ${\bf V}(\Pi)$ est de dimension~$1$, et on est ramen\'e
\`a comprendre le $\PP$-module $(\O_L\point\tA^-)^{H_{\Q_p}}$. Pour que le r\'esultat
soit plus esth\'etique, on tord l'action naturelle de $\PP$ par le caract\`ere $\chi$
donn\'e par $\chi\big(\matrice{a}{b}{0}{1}\big)=a|a|$.  Posons
$$\woe:=(\O_L\point\tA)^{H_{\Q_p}},\quad \woe^+:=(\O_L\point\tA^+)^{H_{\Q_p}},\quad
\woe^-:=\woe/\woe^+,\quad \breve{\O}_L=\O_L\point W(\overline{\bf F}_p)$$
D'apr\`es le dictionnaire d'analyse fonctionnelle $p$-adique~\cite[\S\,IV.3]{mira}, 
on dispose
d'un isomorphisme naturel, $\PP$-\'equivariant,
$$\woe^-\otimes\chi\cong{\cal C}(\Q_p,\O_L)_0$$
o\`u le $0$ en indice signifie ``tend vers $0$ en $\infty$''.
On a alors le r\'esultat suivant:
\begin{theo}\label{ext1}
{\rm (i)}
L'isomorphisme $\woe^-\otimes\chi\cong{\cal C}(\Q_p,\O_L)_0$ peut s'\'etendre, de mani\`ere unique,
en une injection $\PP$-\'equivariante\footnote{Si $W$ est un espace de fonctions sur un ensemble $X$,
on note $\{{\rm Ctes}\}$ le sous-espace des fonctions constantes.}
$$(\O_L\point\tA^-)^{H_{\Q_p}}\otimes\chi\hookrightarrow {\cal C}(\Q_p,\O_L)/\{{\rm Ctes}\}$$
{\rm (ii)} L'injection du {\rm(i)} s'inscrit dans un diagramme commutatif $\PP$-\'equivariant
\[
\xymatrix@R=4mm@C=5mm{0 \ar[r] & \widetilde{\O}_{\cal E}^{-}\otimes\chi\ar[r]\ar[d]^-{\wr}
& (\O_L\point\tA^-)^{H_{\Q_p}}\otimes\chi\ar[r] \ar@{^{(}->}[d]
& H^1({H_{\Q_p}}, \breve{\O}_L\otimes\chi)\ar[r]\ar@{^{(}->}[d] & 0 \\
0\ar[r]& {\cal C}(\Q_p, \O_L)_0\ar[r] 
& \frac{{\cal C}(\Q_p, \O_L)}{\{{\rm Ctes}\}}\ar[r] 
& \frac{{\cal C}(\Q_p, \O_L)}{{\cal C}({\bf P}^1, \O_L)}\ar[r] & 0
}.  \]
\hskip4mm{\rm (iii)} L'image de $H^1({H_{\Q_p}}, \breve{\O}_L\otimes\chi)$ s'identifie
au sous-espace des fonctions continues sur $\widehat{\Q_p^\dual}$ {\rm(pour la topologie profinie)}
modulo celui des fonctions constantes\footnote{Voir~le th.\,\ref{image} et ce qui le pr\'ec\`ede
pour la signification de cet \'enonc\'e.}.

{\rm (iv)} Si $\delta:\Q_p^\dual\to\O_L^\dual$ est un caract\`ere unitaire continu, on dispose
d'une fl\`eche naturelle
$H^1(G_{\Q_p},L(\delta))\to L\otimes_{\O_L}H^1({H_{\Q_p}}, \breve{\O}_L\otimes\chi)$ et l'image
de l'application compos\'ee est:

\quad $\bullet$ $L\cdot\chi^{-1}\delta$ si $\delta\neq 1,\chi$,

\quad $\bullet$ $L\cdot v_p\oplus L\cdot\log$ si $\delta=\chi$,

\quad $\bullet$ $0$ si $\delta=1$.
\end{theo}
\begin{rema}\label{ext2}
L'isomorphisme $H^1({H_{\Q_p}}, \breve{\O}_L))\overset{\sim}{\to}{\cal C}(\widehat{\Q_p^\dual},\O_L)/
{\{{\rm Ctes}\}}$
a l'air naturel, mais la construction le rend assez myst\'erieux
(et la preuve de son existence est non triviale).
Admet-il une construction plus transparente?
\end{rema}

Soit $B(\delta_1,\delta_2):={\rm Ind}_{\BB}^{\GG}(\delta_2\otimes\delta_1\chi^{-1})$.
Alors ${\bf V}(B(\delta_1,\delta_2))=L(\delta_1)$, et donc le foncteur ${\bf V}$ perd
la trace de $\delta_2$.  Mais on a un diagramme commutatif de $\PP$-modules
\[\xymatrix@R=4mm@C=5mm{ 0 \ar[r]  & {\cal C}(\Q_p, \O_L)_0\otimes \delta_1\chi^{-1}\ar[r]\ar@{=}[d]   
&  B(\delta_1, \delta_2)\ar[r] \ar[d] & J(B(\delta_1,\delta_2))\ar[r] \ar[d]& 0 \\
0         \ar[r] & {\cal C}(\Q_p, \O_L)_0\otimes \delta_1\chi^{-1}\ar[r] & (\widetilde{\bf A}^{-}\otimes \delta_1)^{H_{\Q_p}} \ar[r] & H^1({H_{\Q_p}}, \breve{\O}_L\otimes \delta_1)  \ar[r]& 0
}\]
et l'image de $J(B(\delta_1,\delta_2))$ dans $H^1({H_{\Q_p}}, \breve{\O}_L\otimes \delta_1)$
est le sous-espace $H^1(G_{\Q_p},L(\delta_1\delta_2^{-1}))$ si $\delta_1\delta_2^{-1}\neq 1,\chi$
(ce groupe est alors de dimension~$1$).
Autrement dit, l'injection $B(\delta_1,\delta_2)\hookrightarrow(\tA^-\otimes {\bf V}(B(\delta_1,\delta_2)))$
encode une information galoisienne infinit\'esimale.  

\begin{rema}\label{ext3}
(i) Si $\delta_1\delta_2^{-1}=\chi$, au lieu
de $B(\delta_1,\delta_2)$, il faut consid\'erer une extension de $\chi^{-1}\delta_1$ par 
${\rm St}\otimes\chi^{-1}\delta_1$, o\`u ${\rm St}$ est la steinberg, et alors l'information
galoisienne infinit\'esimale encod\'ee est plus subtile car $H^1(G_{\Q_p},L(\delta_1\delta_2^{-1}))$
est de dimension~$2$ (cf.~prop.\,\ref{ext13} et th.\,\ref{catego}).  

(ii) Berger et Vienney~\cite{BV} \'etablissent une bijection entre (mod~$p$) 
repr\'esentations lisses de $\PP$, irr\'eductibles, de dimension infinie, et (mod~$p$) repr\'esentations
irr\'eductibles de $G_{\Q_p}$ (ou, ce qui revient au m\^eme, (mod~$p$) 
$(\varphi,\Gamma)$-modules irr\'eductibles). Les r\'esultats ci-dessus sugg\`erent que, si on sort du
cadre irr\'eductible, les repr\'esentations de dimension finie de $\PP$ ont un r\^ole \`a jouer; il
est possible que la correspondance de~\cite{BV} admette une version cat\'egorifi\'ee (un erzatz
grossier de la correspondance de Langlands locale cat\'egorique~\cite{EGH}).

(iii) Dans le cas pathologique $\delta_1=\delta_2$,
le groupe $H^1(G_{\Q_p},L(\delta_1\delta_2^{-1}))$ meurt dans $H^1({H_{\Q_p}}, \breve{\O}_L\otimes \delta_1)$
et il ne semble pas y avoir d'information galoisienne cach\'ee dans $J(B(\delta_1,\delta_2))$.
\end{rema}

\section{Anneaux de Fontaine}\label{chap1}
\subsubsection*{Le groupe de Galois absolu de $\Q_p$}
 On fixe une cl\^oture alg\'ebrique $\overline{\Q}_p$ de $\Q_p$, 
et on note $\C_p$ le compl\'et\'e $p$-adique de $\overline{\Q}_p$ et $\O_{\C_p}$ son anneau d'entiers.

Soient
$G_{\Q_p}$ le groupe de Galois absolu de $\Q_p$ et 
$\chi: G_{\Q_p}\rightarrow \Z_p^{\dual}$ le caract\`ere cyclotomique.
 Si $F_\infty=\Q_p(\bmu_{p^{\infty}})$, alors on a 
 $H_{\Q_p}={\rm Gal}(\overline{\Q}_p/F_\infty)=\ker(\chi)$, 
ce qui permet de voir $\chi$ aussi comme un isomorphisme de $\Gamma={\rm Gal}(F_\infty/\Q_p)$ sur $\Z_p^{\dual}$;
 on note $a\mapsto \sigma_a$ son inverse.
 
Soit $H_{\Q_p}'\subset H_{\Q_p}$ le groupe de Galois absolu 
de l'extension ab\'elienne maximale $\Q_p^{\rm ab}$ de $\Q_p$.
 Donc $G_{\Q_p}/H_{\Q_p}'= {\rm Gal}(\Q_p^{\rm ab}/\Q_p)$ 
s'identifie au compl\'et\'e profini $\widehat{\Q_p^{\dual}}$
de $\Q_p^{\dual}$ (via cette identification $\chi$ correspond \`a $x\mapsto x|x|$), 
et $H_{\Q_p}/H_{\Q_p}'={\rm Gal}(\overline{\bf F}_p/{\bf F}_p)$.

\subsubsection*{L'anneau $\tA$ et ses sous-objets}
Soit $\widetilde{\E}$ le corps $\C_p^\flat$.
 Il est muni d'une valuation $v_{\E}$.
 On note $\widetilde{\E}^+=\O_{\C_p^\flat}$ l'anneau de ses entiers 
et $\widetilde{\E}^{++}$ l'id\'eal maximal de $\widetilde{\E}^+$.
 On a une d\'ecomposition $\widetilde{\E}^+=\overline{\bf F}_p\oplus \widetilde{\E}^{++}$.

Soient\footnote{L'anneau $\tA^+$ est souvent not\'e $\ainf$.}
 $\tA:=W(\tE)$, $\tA^+:=W(\tE^+)$ et $\tA^{++}:=W(\tE^{++})$.
On a comme ci-dessus une d\'ecomposition $\widetilde{\A}^+=W(\overline{\bf F}_p)\oplus\widetilde{\A}^{++}$.

Les anneaux $\widetilde{\E}$ et $\widetilde{\A}$ sont munis d'actions de $\varphi$ et $G_{\Q_p}$ commutant entre elles, l'action de $\varphi$ \'etant bijective.
Les sous-anneaux $\tE^+$ et $\tA^+$ 
sont stables sous l'action de $\varphi$ et $G_{\Q_p}$;
il en est de m\^eme des id\'eaux  $\tE^{++}$ et $\tA^{++}$.
Cela munit les quotients $\tE^{-}:=\tE/\tE^+$ et $\tA^-:=\tA/\tA^+$ d'actions de $\varphi$ et $G_{\Q_p}$.

\subsubsection*{L'anneau $\tA_{\Q_p}$ et ses sous-objets}
Soient $\epsilon=(1,\zeta_p,\dots)\in\tE^+$ et $\pi:=[\epsilon]-1\in\tA^+$.
On note $\A_{\Q_p}^+$ le sous-anneau $\Z_p[[\pi]]$ de $\tA^+$ et $\A_{\Q_p}$ l'adh\'erence
de $\A_{\Q_p}^+[\frac{1}{\pi}]$ dans~$\tA$. 
On a $\varphi(\pi)=(1+\pi)^p-1$ et $\sigma(\pi)=(1+\pi)^{\chi(\sigma)}-1$, si $\sigma\in G_{\Q_p}$;
on en d\'eduit que $\A_{\Q_p}^+$ et $\A_{\Q_p}$ sont stables par $\varphi$ et $G_{\Q_p}$.

Enfin, soient $\tA_{\Q_p}^+$ et $\tA_{\Q_p}$ les adh\'erences des cl\^otures radicielles
de $\A_{\Q_p}^+$ et $\A_{\Q_p}$ dans $\tA$: si $\tE_{\Q_p}$ est le compl\'et\'e de la cl\^oture
radicielle de ${\bf F}_p((\epsilon-1))$, on a $\tA_{\Q_p}=W(\tE_{\Q_p})$ et
$\tA_{\Q_p}^+=W(\tE_{\Q_p}^+)$.  On a aussi 
$$\tA_{\Q_p}=\tA^{H_{\Q_p}} \quad{\rm et}\quad \tA_{\Q_p}^+=(\tA^+)^{H_{\Q_p}}.$$
Par contre, $(\tA^-)^{H_{\Q_p}}=(\tA/\tA^+)^{H_{\Q_p}}$ est strictement plus gros que $\tA_{\Q_p}/\tA_{\Q_p}^+$
et c'est au lien entre ces deux modules que nous allons nous int\'eresser dans cet article.

\subsubsection*{Extension des scalaires}
Soit $L$ une extension finie de $\Q_p$.
Si $\Lambda$ est une $\Z_p$-alg\`ebre, posons $\O_L\point\Lambda:=\O_L\otimes_{\Z_p}\Lambda$.
Soient $$\oe:=\O_L\point\A_{\Q_p},\quad \oe^+:=\O_L\point\A_{\Q_p}^+,\quad
\woe:=\O_L\point\tA_{\Q_p},\quad \woe^+:=\O_L\point\tA_{\Q_p}^+.$$
Alors $\oe^+=\O_L[[\pi]]$ et $\oe$ est l'anneau des s\'eries de Laurent $\sum_{k\in\Z} a_k\pi^k$,
 avec $a_k\in \O_L$ et $\lim_{k\to-\infty}a_k=0$.
On pose aussi $\oe^-:=\oe/\oe^+$ et $\woe^-:=\woe/\woe^+$. 

On \'etend les actions de $\varphi$ et $G_{\Q_p}$ \`a ces modules par $\O_L$-lin\'earit\'e. On a
$$\woe=(\O_L\point\tA)^{H_{\Q_p}},\quad \woe^+=(\O_L\point\tA^+)^{H_{\Q_p}},\quad
\woe^-\subsetneq(\O_L\point\tA^-)^{H_{\Q_p}}.$$

\subsubsection*{Action du mirabolique}
Notons $\GG$ le groupe ${\bf GL}_2(\Q_p)$, $\BB=\matrice{*}{*}{0}{*}$ son borel sup\'erieur,
 ${\bb P}=\matrice{*}{*}{0}{1}$ son mirabolique, et $\PP^+\subset\PP$ le semi-groupe
$\matrice{\Z_p\moins\{0\}}{\Z_p}{0}{1}$. 
On peut combiner les actions de $G_{\Q_p}$ (qui agit \`a travers son quotient $\Gamma\cong\Z_p^\dual$)
et $\varphi$ sur $\woe$, $\woe^+$ et $\woe^-$
en une action de $\PP$, gr\^ace \`a la formule
\[\matrice{p^ka}{b}{0}{1}\cdot f=[\epsilon^b]\varphi^k(\sigma_a(f)), 
\text{ si } k\in \Z, a\in \Z_p^{\dual} \text{ et } b\in \Q_p, \] 
La restriction de cette action \`a $\PP^+$ laisse stables $\oe$, $\oe^+$ et $\oe^-$.

\section{Analyse $p$-adique}\label{chap2}
\Subsection{Dictionnaire d'analyse fonctionnelle $p$-adique \cite[\S I.1]{gl2}}
\subsubsection{L'analyse sur $\Z_p$}

Si $\mu\in \cdo(\Z_p, \O_L)$ est une mesure sur $\Z_p$ \`a valeurs dans~$\O_L$,
sa transform\'ee d'Amice 
$$A_{\mu}(\pi)=\int_{\Z_p}(1+\pi)^x\mu$$
 appartient \`a $\oe^+$.
Dans l'autre sens, si $f\in \oe$, on d\'efinit $\phi_f: \Z_p\rightarrow L$ par la formule
\[\phi_f(x)=\reso((1+\pi)^{-x}f(\pi)\tfrac{d\pi}{1+\pi}). \]
On munit $\cdo(\Z_p, \O_L)$ et ${\cal C}(\Z_p, \O_L)$ de l'action naturelle 
de $\PP^+$:
\begin{align*}
\int_{\Z_p}\phi(x)\,\matrice{a}{b}{c}{d}\mu:=\int_{\Z_p}\phi(ax+b)\,\mu\\
(\matrice{p^ka}{b}{0}{1}\cdot\phi)(x)=\begin{cases}\phi(\frac{x-b}{p^ka}) & \text{si } x\in b+p^k\Z_p; \\ 
0 & {\text{sinon.}}\end{cases}
\end{align*}
On a alors le r\'esultat suivant qui fournit un dictionnaire entre l'analyse fonctionnelle sur $\Z_p$
et l'anneau $\oe$ (cf.~\cite[th.\,0.1]{gl2} -- et~\cite[prop.\,1.3]{durham} pour la torsion par $\chi^{-1}$).
\begin{prop}\label{ext4}
{\rm (i)} $\mu\mapsto A_\mu$ induit un isomorphisme $\PP^+$-\'equivariant 
$$\cdo(\Z_p, \O_L)\overset{\sim}{\to}\oe^+.$$

{\rm (ii)} $f\mapsto\phi_f$ induit isomorphisme $\PP^+$-\'equivariant
$$\oe^-\overset{\sim}{\to}{\cal C}(\Z_p,\O_L)\otimes\chi^{-1}$$
o\`u $\chi(x)=x|x|$, vu comme caract\`ere de $\PP^+$ par $\chi\big(\matrice{a}{b}{0}{1}\big):=\chi(a)$.
\end{prop}

\subsubsection{L'analyse sur $\Q_p$}
Le dictionnaire d'analyse fonctionnelle $p$-adique s'\'etend \`a $\Q_p$, voir \cite[IV.3]{mira}. 
On note $\cdo(\Q_p, \O_L)_{\rm pc}$ l'espace des mesures sur $\Q_p$
``nulles \`a l'infini'' (c'est aussi le $\O_L$-dual de l'espace des fonctions uniform\'ement
continues sur $\Q_p$, \`a valeurs dans $\O_L$) et
${\cal C}(\Q_p, \O_L)_0$ l'espace des fonctions continues sur $\Q_p$ tendant vers $0$ \`a l'infini.
On munit ces espaces d'actions de $\PP$
en posant
$$\int_{\Q_p}\phi(x)\,\matrice{a}{b}{0}{1}\mu:=\int_{\Q_p}\phi(ax+b)\,\mu,
\quad \matrice{a}{b}{0}{1}\phi(x)=\phi\big(\tfrac{x-b}{a}\big)$$
On a alors le r\'esultat suivant qui \'etend le dictionnaire de la prop.\,\ref{ext4} \`a $\Q_p$.
\begin{prop}\label{ext5}
{\rm (i)}
La transform\'ee de Fourier $\mu\mapsto \int_{\Q_p}[\epsilon^x]\mu$ 
induit un isomorphisme $\PP$-\'equivariant
$$\cdo(\Q_p, \O_L)_{\rm pc}\overset{\sim}{\to}\woe^+$$

{\rm (ii)} La transform\'ee $z\mapsto\phi_z$, o\`u $\phi_z(x):=\reso([\epsilon^{-x}]z\tfrac{d\pi}{1+\pi})$
et $\reso$ est l'unique extension continue de $\reso$ \`a $\woe\frac{d\pi}{1+\pi}$ v\'erifiant 
$\reso(\varphi(z)\frac{d\pi}{1+\pi})=\reso(z\frac{d\pi}{1+\pi})$,
induit un isomorphisme $\PP$-\'equivariant:
$$\widetilde{\O}_{\cal E}^{-}\cong {\cal C}(\Q_p, \O_L)_0\otimes\chi^{-1}$$
\end{prop}

\Subsection{Extension du dictionnaire d'analyse fonctionnelle $p$-adique}
La prop.\,\ref{ext5} fournit une interpr\'etation de $\woe^-$ en termes de fonctions
continues sur $\Q_p$ avec un comportement sp\'ecial \`a l'infini; nous allons
donner une telle interpr\'etation pour $(\O_L\point\tA^-)^{H_{\Q_p}}$ qui contient $\woe^-$.
Dans ce paragraphe, on fabrique un plongement naturel dans les fonctions continues sur $\Q_p$,
et au chapitre~\ref{chap4}, on d\'ecrit l'image de ce plongement.

\subsubsection{D\'evissage de $(\O_L\point\tA^-)^{H_{\Q_p}}$}
On pose
$$\breve{\O}_L:=\O_L\point W(\overline{\bf F}_p),\quad \breve{L}:=\breve{\O}_L[\tfrac{1}{p}]$$
Ces anneaux sont munis d'actions $\O_L$-lin\'eaires de $\varphi$ et $G_{\Q_p}$ qui commutent.
 
 \begin{lemm}\label{ext6}
 {\rm (i)} $H^1({H_{\Q_p}}, \O_L\point\widetilde{\A})=0$ et $H^1({H_{\Q_p}}, \O_L\point\widetilde{\A}^{++})=0$.

{\rm (ii)} On a une suite exacte naturelle 
\[0\rightarrow \woe^-\rightarrow (\O_L\point\tA^-)^{H_{\Q_p}}\rightarrow H^1({H_{\Q_p}}, \breve{\O}_L)\rightarrow 0.\]
\end{lemm}
\begin{proof}
Le {\rm (i)} est classique (descente presque \'etale).
 Le {\rm (ii)} s'en d\'eduit en utilisant la suite exacte longue de cohomologie associ\'ee \`a la suite exacte 
\[0\rightarrow \O_L\point\tA^+\rightarrow \O_L\point\tA\rightarrow \O_L\point\tA^{-}\rightarrow 0,\]
 le {\rm (i)} et la decomposition $\O_L\point\tA^+=\breve\O_L\oplus \O_L\point\tA^{++}$. 
\end{proof}

\subsubsection{Une injection dans les fonctions continues}
Soit $\UU:=\matrice{1}{\Q_p}{0}{1}$.

\begin{lemm}\label{ext7}
 $H^0(\UU, {\cal C}(\Q_p, \O_L))=\O_L$ et $H^1(\UU, {\cal C}(\Q_p, \O_L))=0$.  
\end{lemm}
\begin{proof}
 L'espace ${\cal C}(\Q_p, \O_L)$, vu comme repr\'esentation de $\UU\cong \Q_p$, 
est l'induite de $1$ \`a $\UU$ de la repr\'esentation triviale;
 le r\'esultat s'en d\'eduit via le lemme de Shapiro. 

Plus explicitement, le r\'esultat pour $H^0$ est imm\'ediat (une fonction sur $\Q_p$ invariante
par translation par $\Q_p$ est constante).  Pour $H^1$, notons $U_N$ le sous-groupe
$\matrice{1}{p^{-N}\Z_p}{0}{1}$ de $\UU$. Commen\c{c}ons par remarquer que
$H^1(U_0,{\cal C}(\Z_p,\O_L))=0$ (cela revient \`a v\'erifier que toute fonction continue $\phi$ sur $\Z_p$
peut s'\'ecrire sous la forme $\tilde\phi(x)-\tilde\phi(x-1)$, avec $\tilde\phi$ continue,
 ce qui est imm\'ediat sur le d\'eveloppement de Mahler en termes de polyn\^omes binomiaux).
Comme ${\cal C}(\Q_p,\O_L)\cong {\cal C}(\Z_p,\O_L)^{\Q_p/\Z_p}$ comme $U_0$-module,
on a aussi $H^1(U_0,{\cal C}(\Q_p,\O_L))=0$ et, par suite, $H^1(U_N,{\cal C}(\Q_p,\O_L))=0$
pour tout $N\geq 0$.  

Soit alors $b\mapsto\phi_b$ un $1$-cocycle sur $\Q_p$ \`a valeurs
dans ${\cal C}(\Q_p,\O_L)$. Il r\'esulte de ce qui pr\'ec\`ede que, pour tout $N\geq 0$, il existe
$\phi^N$ v\'erifiant $(b-1)\cdot \phi^N=\phi_b$, pour tout $b\in U_N$, et on peut imposer
que $\phi^N(0)=0$ auquel cas, $\phi^N$ est uniquement d\'etermin\'ee sur $p^{-N}\Z_p$
(et \`a constante pr\`es sur chaque $a+p^{-N}\Z_p$).  Cette unicit\'e fait que $\phi^{N+1}=\phi^N$
sur $p^{-N}\Z_p$, et donc que $\phi=\lim_N\phi^N$ existe. On a alors $(b-1)\cdot\phi=\phi_b$
pour tout $b\in \UU$ puisque c'est vrai en restriction \`a $p^{-N}\Z_p$, pour tout $N$ et
tout $b\in U_N$.  Notre cocycle est donc un cobord, ce qui prouve la nullit\'e du $H^1$.
\end{proof}

\begin{prop}\label{ext8}
Il existe un unique plongement ${\bb P}$-\'equivariant
\[\iota:(\O_L\point\tA^-)^{H_{\Q_p}}\hookrightarrow ({\cal C}(\Q_p, \O_L)/{\{{\rm Ctes}\}})\otimes \chi^{-1}\]
rendant commutatif le diagramme
\[\xymatrix@R=4mm@C=5mm{ 
\woe^-\ar[r]\ar[d]& (\O_L\point\tA^-)^{H_{\Q_p}}\ar[d] \\
{\cal C}(\Q_p, \O_L)_0\otimes\chi^{-1}\ar[r] &({\cal C}(\Q_p, \O_L)/{\{{\rm Ctes}\}})\otimes\chi^{-1} }.\]
\end{prop}
\begin{proof}
On part de la suite exacte
\[0\rightarrow \woe^-\rightarrow(\O_L\point\tA^-)^{H_{\Q_p}}\rightarrow H^1({H_{\Q_p}}, \O_L\point\tA^+)\rightarrow 0.\]
Par ailleurs, $H^1({H_{\Q_p}}, \O_L\point\tA^{++})=0$, et donc $H^1({H_{\Q_p}}, \O_L\point\tA^+)$ 
est tu\'e par $\matrice{1}{b}{0}{1}-1$ pour tout $b\in \Q_p$.
 Il s'ensuit que, si $v\in (\O_L\point\tA^{-})^{H_{\Q_p}}$, 
alors $v_b:=\big(\matrice{1}{b}{0}{1}-1\big)\cdot v\in \woe^-$.
 Posant $\phi_b=\phi_{v_b}$, on obtient de la sorte un $1$-cocycle sur $\UU$ \`a valeurs dans ${\cal C}(\Q_p, \O_L)_0\otimes\chi^{-1}$.
 On d\'eduit du lemme pr\'ec\'edent que ce $1$-cocycle se trivialise dans ${\cal C}(\Q_p, \O_L)\otimes \chi^{-1}$, de mani\`ere unique modulo les constantes.
Autrement dit,
il existe un unique $\phi\in ({\cal C}(\Q_p, \O_L)/{\{{\rm Ctes}\}})\otimes \chi^{-1}$ 
tel que $\phi_b= ((\begin{smallmatrix}1& b\\ 0 & 1\end{smallmatrix})-1)\cdot \phi$ pour tout $b$, 
et le plongement 
$(\O_L\point\tA^{-})^{H_{\Q_p}}\hookrightarrow ({\cal C}(\Q_p, \O_L)/{\{{\rm Ctes}\}})\otimes \chi^{-1}$ 
est donn\'e par $v\mapsto \phi$ (la fl\`eche
$v\mapsto \phi$ est injective: si $\phi=0$, on a $v_b=0$ pour tout $b$; on en d\'eduit
que, si $\hat v\in \O_L\point\tA$ est un rel\`evement de $v$, alors $([\epsilon^b]-1)\hat v\in \O_L\point\tA^+$,
pour tout $b$, et donc $\hat v\in \O_L\point\tA^+$ (lemme~\ref{truc} ci-dessous) et $v=0$). 

Il r\'esulte de ce qui pr\'ec\`ede que $\iota$ a un unique prolongement $\UU$-\'equivariant
\`a $(\O_L\point\tA^-)^{H_{\Q_p}}$. Cette unicit\'e implique que ce prolongement est $\PP$-\'equivariant:
si $g\in\PP$ et si $u\in\UU$, alors $u^{-1}g^{-1}\iota gu=g^{-1}(gug^{-1})^{-1}\iota(gug^{-1})g=
g^{-1}\iota g$ puisque $gug^{-1}\in\UU$;
 il s'ensuit que $g\iota g^{-1}$ est $\UU$-\'equivariant, et
comme il co\"{\i}ncide avec $\iota$ sur $\woe^-$ (puisque $\iota$ est $\PP$ \'equivariant
sur $\woe^-$), on a $g\iota g^{-1}=\iota$, ce qui permet de conclure.
\end{proof}

\begin{lemm}\label{truc}
Si $x\in \O_L\point\tA$ v\'erifie $([\epsilon^b]-1)x\in \O_L\point\tA^+$,
pour tout $b\in\Q_p$, alors $x\in \O_L\point\tA^+$.
\end{lemm}
\begin{proof}
En d\'ecomposant tout sur une base de $\O_L$ sur $\Z_p$, on se ram\`ene \`a prouver
que, si $x\in\tA$ et si $([\epsilon^b]-1)x\in \tA^+$ pour tout $b\in\Q_p$, alors $x\in \tA^+$.
Pour cela, on \'ecrit $x=\sum_{n\geq 0}p^n[x_n]$, avec $x_n\in\tE$, et on doit prouver
que $x_n\in\tE^+$.  En r\'eduisant modulo~$p$, on obtient $v_\E((\epsilon^b-1)x_0)\geq 0$ pour tout $b$,
et donc (en prenant $b=p^{-k}$), $v_\E(x_0)\geq -\frac{1}{(p-1)p^{k-1}}$ pour tout $k$,
et donc $v_\E(x_0)\geq 0$, i.e.~$x_0\in \tE^+$.

Posons $y_1=\frac{1}{p}(x-p[x_0])=\sum_{n\geq 0}p^n[x_{n+1}]$. 
On a $([\epsilon^b]-1)y_1\in\tA\cap\frac{1}{p}\tA^+= \tA^+$, pour tout $b$, et donc $x_1\in\tE^+$.
Une r\'ecurrence imm\'ediate permet de conclure.
\end{proof}

\subsubsection{Un diagramme commutatif}
Soit ${\bf P}^1=\Q_p\cup\{\infty\}$.
On a une d\'ecomposition $\PP$-\'equivariante
$${\cal C}(\piqp,\O_L)={\cal C}(\Q_p,\O_L)_0\oplus \O_L{\bf 1}_{\piqp},\quad
\phi\mapsto (\phi-\phi(\infty))+\phi(\infty){\bf 1}_{\piqp}$$
o\`u ${\bf 1}_{\piqp}$ est la fonction constante sur $\piqp$, de valeur $1$.
On en d\'eduit une suite exacte $\PP$-\'equivariante
\[
\xymatrix@C=4mm{
0\ar[r]& {\cal C}(\Q_p,\O_L)_0\ar[r] & \frac{{\cal C}(\Q_p, \O_L)}{\{{\rm Ctes}\}}\ar[r] 
&  \frac{ {\cal C}(\Q_p, \O_L)}{{\cal C}(\piqp,\O_L)_0}\ar[r]& 0 \\
 }.
\]
D'o\`u un diagramme ${\bb P}$-\'equivariant
\begin{equation}\label{diagfond}
\xymatrix@R=4mm@C=5mm{0 \ar[r] & \widetilde{\O}_{\cal E}^{-} \ar[r]\ar[d]^-{\wr}
& (\O_L\point\tA^-)^{H_{\Q_p}}\ar[r] \ar@{^{(}->}[d]
& H^1({H_{\Q_p}}, \breve{\O}_L)\ar[r]\ar@{^{(}->}[d] & 0 \\
0\ar[r]& {\cal C}(\Q_p, \O_L)_0\otimes\chi^{-1}\ar[r] 
& \frac{{\cal C}(\Q_p, \O_L)}{\{{\rm Ctes}\}}\otimes\chi^{-1}\ar[r] 
& \frac{{\cal C}(\Q_p, \O_L)}{{\cal C}({\bf P}^1, \O_L)}\otimes\chi^{-1}\ar[r] & 0
}
\end{equation}
Nous allons d\'eterminer l'image de $H^1({H_{\Q_p}}, \breve{\O}_L)\to
\frac{{\cal C}(\Q_p, \O_L)}{{\cal C}({\bf P}^1, \O_L)}\otimes\chi^{-1}$,
ce qui fournira une description de celle de $(\O_L\point\tA^-)^{H_{\Q_p}}$.

\section{Le $\PP$-module $H^1({H_{\Q_p}}, \breve{\O}_L)$}\label{chap3}
\Subsection{Relation avec $H^1(H_{\Q_p}',\O_L)$}
On a 
$$G_{\Q_p}/H_{\Q_p}'=(G_{\Q_p}/{H_{\Q_p}})\times({H_{\Q_p}}/H_{\Q_p}'),\quad
G_{\Q_p}/{H_{\Q_p}}\cong\Z_p^\dual,\quad {H_{\Q_p}}/H_{\Q_p}'=\sigma_p^{\cZ}$$
Le groupe $\sigma_p^{\cZ}\times\Z_p^\dual\times\varphi^\Z$ agit sur:

$\bullet$ $\breve{\O}_L$ \`a travers $\sigma_p^{\cZ}\times\varphi^\Z$ (et les actions
de $\sigma_p$ et $\varphi$ co\"{\i}ncident; en particulier, l'action de $\varphi^\Z$
s'\'etend par continuit\'e en une action de $\varphi^{\cZ}$).

$\bullet$ $H^1({H_{\Q_p}}, \breve{\O}_L)$ \`a travers $\Z_p^\dual\times\varphi^\Z$,
o\`u $\varphi$ agit sur $\breve{\O}_L$ et $\Z_p^\dual$ agit comme
$G_{\Q_p}/{H_{\Q_p}}$ ($\breve{\O}_L$ est muni d'une action de $G_{\Q_p}$).

$\bullet$ $H^1(H_{\Q_p}', \O_L)$ \`a travers $\sigma_p^{\cZ}\times\Z_p^\dual=G_{\Q_p}/H_{\Q_p}'$.

On identifie le sous-groupe $\{1\}\times\Z_p^\dual\times \varphi^\Z$ \`a $\Q_p^\dual$
en envoyant $\varphi$ sur $p$, et le sous-groupe
$\sigma_p^{\widehat{\Z}}\times\Z_p^\dual\times\{1\}$ \`a $\widehat{\Q_p^\dual}$ en envoyant
$\sigma_p$ sur $p^{-1}$ (on obtient ainsi l'identification $G_{\Q_p}^{\rm ab}=\widehat{\Q_p^\dual}$
de la th\'eorie locale du corps de classe).

\begin{lemm}\label{iso1}
 {\rm (i)} La surjection $(\O_L\point\widetilde{\A}^{-})^{H_{\Q_p}}\rightarrow H^1({H_{\Q_p}}, \breve{\O}_L)$ 
est $\Q_p^{\dual}$-\'equivariante. 

{\rm (ii)} On a un isomorphisme naturel de $\sigma_p^{\cZ}\times\Z_p^\dual\times\varphi^\Z$-modules
\[H^1({H_{\Q_p}}, \breve{\O}_L)\widehat{\otimes}_{\O_L} \breve{\O}_L
\cong H^1(H_{\Q_p}', \O_L)\widehat{\otimes}_{\O_L}\breve{\O}_L. \]
En particulier,
\begin{align*}
H^1(H'_{\Q_p},\O_L)&\cong (H^1({H_{\Q_p}}, \breve{\O}_L)\widehat{\otimes}_{\O_L} \breve{\O}_L)^{\varphi=1}\\
H^1({H_{\Q_p}}, \breve{\O}_L)
&\cong (H^1(H_{\Q_p}', \O_L)\widehat{\otimes}_{\O_L}\breve{\O}_L)^{\sigma_p^{\cZ}}
\end{align*}
\end{lemm}
\begin{proof} 
{\rm (i)} La fl\`eche $(\O_L\point\widetilde{\A}^{-})^{H_{\Q_p}}\rightarrow H^1({H_{\Q_p}}, \breve{\O}_L)$ commute \`a $\varphi$,
puisque $\varphi$ commute \`a $G_{\Q_p}$.
Par ailleurs, elle est
$G_{\Q_p}$-\'equivariante par d\'efinition de l'application de connexion, et donc $\Z_p^\dual$-\'equivariante
puisque $G_{\Q_p}$ agit \`a travers $G_{\Q_p}/{H_{\Q_p}}\hskip.5mm{=}\hskip.5mm\Z_p^\dual$.

{\rm (ii)} Par descente \'etale, on a $H^i({H_{\Q_p}}/H_{\Q_p}', \breve{\O}_L)=0$, si $i\geq 1$.
 La suite d'inflation-restriction et la nullit\'e de $H^i({H_{\Q_p}}/H_{\Q_p}', \breve{\O}_L)$,
pour $i=1,2$, donnent un isomorphisme 
\[ H^1({H_{\Q_p}}, \breve{\O}_L)\cong (H^1(H_{\Q_p}',\Z_p)\widehat{\otimes} \breve{\O}_L)^{{H_{\Q_p}}/H_{\Q_p}'},\]
o\`u ${H_{\Q_p}}/H_{\Q_p}'\cong \sigma_p^{\widehat{\Z}}$ agit diagonalement.
 Le r\'esultat s'en d\'eduit via Hilbert $90$ dont on d\'eduit par d\'evissage que, si $V$ est 
une $\O_L$-repr\'esentation de ${H_{\Q_p}}/H_{\Q_p}'$, l'application naturelle 
\[\breve{\O}_L\widehat{\otimes}_{\O_L} (\breve{\O}_L\widehat{\otimes}_{\O_L}V)^{{H_{\Q_p}}/H_{\Q_p}'}
\rightarrow \breve{\O}_L\widehat{\otimes}_{\O_L}V\]
est un isomorphisme.  
\end{proof}

\Subsection{Sous-espaces propres pour l'action de $\Q_p^\dual$}
Rappelons que, si $\delta:\Q_p^\dual\to \O_{L}^\dual$ est un caract\`ere unitaire, alors
$\delta(\varphi)=\delta(p)$ mais $\delta(\sigma_p)=\delta(p)^{-1}$. 
Notons
$e_\delta$ la base $1\otimes\delta$ du $(\varphi,\Gamma)$-module ${\O}_L\otimes\delta$.
On a alors $\breve{\O}_{L}\otimes\delta=\breve{\O}_{L}\cdot e_\delta$ en tant que $(\varphi,\Gamma)$-module
(ou $\Q_p^\dual$-module) sur $\breve{\O}_L$.

Notons ${\bf V}(\delta)$ la $\O_L$-repr\'esentation 
$(\O_L\point\tA\otimes\delta)^{\varphi=1}=(\breve{\O}_L\cdot e_\delta)^{\varphi=1}$ de $G_{\Q_p}$; 
le choix
d'une base (i.e.~de $\alpha_\delta\in\breve{\O}_{L}^\dual$
v\'erifiant $\varphi(\alpha_\delta)=\delta(p)^{-1}\alpha_\delta$)
fournit un isomorphisme $\O_L(\delta)\cong {\bf V}(\delta)$.

\begin{lemm}\label{ext9}
On a un isomorphisme:
$$ H^1(H_{\Q_p}',{\bf V}(\delta))^{G_{\Q_p}/H_{\Q_p}'}\overset{\sim}{\to}
(H^1({H_{\Q_p}},\breve{\O}_{L})\otimes e_{\delta})^{\Q_p^\dual}.$$
Plus g\'en\'eralement, si $\delta_1,\delta_2:\Q_p^\dual\to\O_L^\dual$ sont des caract\`eres continus,
$$ H^1(H_{\Q_p}',{\bf V}(\delta_1\delta_2))^{G_{\Q_p}/H_{\Q_p}'}\overset{\sim}{\to}
(H^1({H_{\Q_p}},\breve{\O}_{L}\otimes\delta_1)\otimes e_{\delta_2})^{\Q_p^\dual}.$$
\end{lemm}
\begin{proof}
Cela r\'esulte du diagramme commutatif suivant:
$$\xymatrix@R=4mm@C=3mm{
(H^1({H_{\Q_p}},\breve{\O}_{L})\otimes e_\delta)^{\Q_p^\dual}\ar@{=}[r]
&(H^1({H_{\Q_p}},\breve{\O}_{L})\otimes e_\delta)^{\Z_p^\dual\times\varphi^\Z}\ar[d]^-{\wr}\\
H^1(H_{\Q_p}',{\bf V}(\delta))^{G_{\Q_p}/H_{\Q_p}'}\ar@{=}[r]
&(H^1(H_{\Q_p}',\O_L)\wotimes(\breve{\O}_{L}\otimes e_\delta))^{\Z_p^\dual\times \sigma_p^{\cZ}\times\varphi^\Z}
}$$
Dans ce diagramme:

$\bullet$  L'\'egalit\'e de la premi\`ere ligne
est juste l'identification $\Q_p^\dual=\Z_p^\dual\times\varphi^\Z$.

$\bullet$ L'isomorphisme vertical d\'ecoule du lemme~\ref{iso1}.

$\bullet$ On a $H^1(H_{\Q_p}',{\bf V}(\delta))=H^1(H_{\Q_p}',\O_L)\otimes{\bf V}(\delta)$ car $H_{\Q_p}'$
agit trivialement sur ${\bf V}(\delta)$, et
l'isomorphisme ${\bf V}(\delta)\cong (\breve{\O}_{L}\otimes e_\delta)^{\varphi=1}$
est la d\'efinition de ${\bf V}(\delta)$ (cf.~ci-dessus).
\end{proof}

\begin{lemm}\label{ext10}
{\rm (i)} On a:
 
$\bullet$ $H^0(G_{\Q_p},{\bf V}(\delta))\otimes_{\O_L}L=0$ sauf si $\delta=1$ o\`u il est de dimension~$1$.

$\bullet$ $H^1(G_{\Q_p},{\bf V}(\delta))\otimes_{\O_L}L$ est de dimension~$1$ sauf si $\delta=1,\chi$ o\`u il est
de dimension~$2$.

$\bullet$ $H^2(G_{\Q_p},{\bf V}(\delta))\otimes_{\O_L}L=0$ sauf si $\delta=\chi$ o\`u il est de dimension~$1$.

{\rm (ii)} La fl\`eche 
\[H^1(G_{\Q_p},{\bf V}(\delta))\otimes_{\O_L}L\to H^1(H'_{\Q_p},{\bf V}(\delta)\otimes_{\O_L}L
)^{G_{\Q_p}/H'_{\Q_p}}\to (H^1({H_{\Q_p}},\breve{\O}_{L})\otimes_{\O_L}L e_\delta))^{\Q_p^\dual}\]
d\'eduite du lemme~\ref{ext9} est un isomorphisme sauf si $\delta=1$ o\`u elle est identiquement nulle.
\end{lemm}
\begin{proof}
Le (i) est standard (la preuve utilise la formule d'Euler-Poincar\'e de Tate et la dualit\'e
de Tate pour calculer la dimension des $H^2$ en fonction de celle des $H^0$, ou bien on
peut utiliser la th\'eorie des $(\varphi,\Gamma)$-modules~\cite[chap.\,2]{gaetan} ou \cite[th.\,1.38]{durham}).

Passons au (ii).
La suite spectrale de Hochschild-Serre pour l'extension de groupes 
\[1\rightarrow H_{\Q_p}'\rightarrow G_{\Q_p}\rightarrow \widehat{\Q_p^{\dual}}\rightarrow 1\]
nous donne une suite exacte (on note $\delta$ la $L$-repr\'esentation $L(\delta)$ de $G_{\Q_p}$)
\begin{align}\label{Hosh}
0\rightarrow H^1(\widehat{\Q^{\dual}_p}, \delta)\rightarrow H^1(G_{\Q_p},  \delta)\rightarrow H^1(H_{\Q_p}', \delta)^{\widehat{\Q_p^{\dual}}}\rightarrow  H^2(\widehat{\Q_p^{\dual}}, \delta)\rightarrow H^2(G_{\Q_p}, \delta),\end{align}
Maintenant $\widehat{\Q^{\dual}_p}\cong\Z_p^2\times U$, o\`u $U$ est profini d'ordre premier \`a $p$ 
(\`a $\Z/2\Z$ pr\`es si $p=2$); si on \'ecrit un \'el\'ement de $\Z_p^2\times U$ sous
la forme $(x_1,x_2,u)$, avec $x_1,x_2\in\Z_p$ et $u\in U$, cela permet de factoriser $\delta$
sous la forme $\delta_1(x_1)\delta_2(x_2)\delta_U(u)$, o\`u $\delta_1,\delta_2$ sont des
caract\`eres de $\Z_p$ et $\delta_U$ un caract\`ere de $U$.

Maintenant, $H^i(U,\delta_U)=0$ sauf si $\delta_U=1$ et $i=0$, o\`u ce groupe est de dimension~$1$,
et $H^i(\Z_p,\delta_j)=0$ sauf si $\delta_j= 1$ et $i=0,1$ o\`u ce groupe est de dimension~$1$.
Il s'ensuit, via la formule de K\"unneth ou un d\'evissage utilisant la
suite spectrale de Hochschild-Serre, que:

$\bullet$
$H^i(\widehat{\Q^{\dual}_p}, \delta)=0$, pour tout $i$, si $\delta\neq 1$, ce qui prouve le (ii)
dans ce cas-l\`a.  

$\bullet$ Si $\delta=1$, on a $\dim_LH^1(\widehat{\Q^{\dual}_p}, \delta)=2$
et $\dim_LH^2(\widehat{\Q^{\dual}_p}, \delta)=1$.
On conclut gr\^ace \`a 
ces calculs de dimension et ceux du (i) que la fl\`eche est effectivement nulle dans le cas $\delta=1$.
\end{proof}

\begin{rema}\label{ext11}
Dans le cas pathologique $\delta=1$, la suite exacte~(\ref{Hosh}) montre que
$(H^1({H_{\Q_p}},\breve{L})\otimes e_\delta)^{\Q_p^\dual}\to H^2(\widehat{\Q_p^{\dual}}, \delta)$ est
un isomorphisme, et donc $(H^1({H_{\Q_p}},\breve{L})\otimes e_\delta)^{\Q_p^\dual}$ est de dimension~$1$
mais ne semble avoir aucun lien avec la cohomologie de $G_{\Q_p}$. 
\end{rema}

\section{Description de l'image}\label{chap4}
\Subsection{Vecteurs $\widehat{\Q_p^\dual}$-continus}
\subsubsection{Fonctions presque p\'eriodiques \`a l'infini}
On dit que $\phi\in{\cal C}(\Q_p,\O_L/p^n)$ est {\it p\'eriodique \`a l'infini} si il existe $k\geq 1$
tel que $\phi(x)=\phi(p^{-k}x)$ pour tout $x$ assez grand (i.e., tout $x$ v\'erifiant $v_p(x)\leq -N$ pour $N\in \N$ assez grand).
On dit que $\phi\in{\cal C}(\Q_p,\O_L)$ est {\it presque p\'eriodique \`a l'infini} si 
son image dans ${\cal C}(\Q_p,\O_L/p^n)$
est p\'eriodique
\`a l'infini, pour tout $n$.

On note ${\cal C}^{\rm pp}(\Q_p,\O_L)$ l'espace des fonctions continues, presque p\'eriodiques \`a l'infini
et ${\cal C}^{\rm pp}(\Q_p,\O_L/p^n)$ celui des fonctions continues (et donc localement constantes),
p\'eriodiques \`a l'infini.  
On a ${\cal C}^{\rm pp}(\Q_p,\O_L)=\varprojlim_n{\cal C}^{\rm pp}(\Q_p,\O_L/p^n)$.
Notons que, si $\phi\in {\cal C}(\widehat{\Q_p^{\dual}}, \O_L)$, alors
${\bf 1}_{\Q_p\moins\Z_p}\phi\in {\cal C}^{\rm pp}(\Q_p,\O_L)$ par d\'efinition
de la compl\'etion profinie.

\begin{lemm}\label{ppp}
{\rm (i)} Tout $\phi\in {\cal C}^{\rm pp}(\Q_p,\O_L)$ peut s'\'ecrire, de mani\`ere unique, sous la forme
$\phi_0+{\bf 1}_{\Q_p\moins\Z_p}\phi_\infty$, avec $\phi_0\in{\cal C}(\Q_p,\O_L)_0$
et $\phi_\infty\in {\cal C}(\widehat{\Q_p^{\dual}}, \O_L)$.

{\rm (ii)} ${\cal C}^{\rm pp}(\Q_p,\O_L)$ est stable par $\PP$ et $\phi\mapsto\phi_\infty$ induit une suite
exacte
$$0\to {\cal C}(\Q_p,\O_L)_0\to {\cal C}^{\rm pp}(\Q_p,\O_L)\to {\cal C}(\widehat{\Q_p^{\dual}}, \O_L)\to 0$$
$\PP$-\'equivariante, l'action de $\UU$ sur ${\cal C}(\widehat{\Q_p^{\dual}}, \O_L)$ \'etant triviale.
\end{lemm}
\begin{proof}
Il suffit de prouver l'\'enonc\'e \'equivalent pour ${\cal C}^{\rm pp}(\Q_p,\O_L/p^n)$ 
et de passer \`a la limite. Soit donc $\phi\in {\cal C}^{\rm pp}(\Q_p,\O_L/p^n)$. Par d\'efinition,
il existe $\phi_\infty:\Q_p^\dual\to\O_L/p^n$, p\'eriodique (i.e.~il existe $k\geq 1$ tel que
$\phi_\infty(p^{-k}x)=\phi_\infty(x)$ pour tout $x\in\Q_p^\dual$) 
co\"{\i}ncidant avec $\phi$ pour $x$
assez grand; $\phi_\infty$ est uniquement d\'etermin\'ee et $\phi-{\bf1}_{\Q_p\moins\Z_p}\phi_\infty$
est \`a support compact. 

Ceci prouve le (i); le (ii) r\'esulte de ce que, si
$\phi\in {\cal C}^{\rm pp}(\Q_p,\O_L/p^n)$, et si $b\in\Q_p$, alors
$\phi(x-b)-\phi(x)$ est identiquement nul pour $x$ assez grand 
(si $\phi(p^{-k}x)=\phi(x)$ pour $x$ assez grand, et si $v_p(x)=-kq+r$, 
avec $0\leq r\leq q-1$, on a $\phi(x-b)-\phi(x)=\phi(p^{kq}x-p^{kq}b)-\phi(p^{kq}x)$,
et les $p^{kq}x$ varient dans un compact tandis que $p^{kq}b\to 0$).
\end{proof}

\begin{coro}\label{image1}
On a des fl\`eches $\Q_p^{\dual}$-\'equivariantes
\[  \frac{{\cal C}(\widehat{\Q_p^{\dual}}, \O_L)}{\{{\rm Ctes}\}}\overset{\sim}{\leftarrow}
\frac{{\cal C}^{\rm pp}(\Q_p, \O_L)}{{\cal C}({\bf P}^1, \O_L)}
\hookrightarrow  \frac{{\cal C}(\Q_p, \O_L)}{{\cal C}({\bf P}^1, \O_L)}.\]
\end{coro}
\begin{proof}
C'est une cons\'equence du lemme~\ref{ppp} et de la d\'ecomposition
${\cal C}(\piqp,\O_L)=\O_L{\bf 1}_{\piqp}\oplus {\cal C}(\Q_p,\O_L)_0$.
\end{proof}

\subsubsection{Presque p\'eriodicit\'e et $\widehat{\Q_p^\dual}$-continuit\'e}
Si $V$ est un $\O_L$-module s\'epar\'e et complet pour la topologie $p$-adique, muni d'une
action continue de $\Q_p^\dual$, on dit que $v\in V$ est {\it $\widehat{\Q_p^\dual}$-continu} si
la fonction
$\gamma\mapsto\gamma\star v$, de $\Q_p^\dual$ dans $V$,
s'\'etend en une fonction continue sur $\widehat{\Q_p^\dual}$.
De mani\`ere \'equivalente, $v$ est $\widehat{\Q_p^\dual}$-continu si, pour tout $n\geq 1$,
$k\mapsto p^k\star v$ est p\'eriodique vue comme fonction de $\Z$ dans $V/p^n V$.
\begin{lemm}\label{image2}
L'injection du cor.\,\ref{image1} identifie $\frac{{\cal C}(\widehat{\Q_p^{\dual}}, \O_L)}{\{{\rm Ctes}\}}$ au sous-espace
des vecteurs $\widehat{\Q_p^\dual}$-continus de $\frac{{\cal C}(\Q_p, \O_L)}{{\cal C}({\bf P}^1, \O_L)}$.
\end{lemm}
\begin{proof}
Il s'agit de comprendre quels sont les vecteurs invariants de 
$V_n:=\frac{{\cal C}(\Q_p\moins \Z_p, \O_L/p^n)}{{\cal C}({\bf P}^1\moins\Z_p, \O_L/p^n)}$ sous
l'action du semi-groupe $p^{k\N}$, o\`u $p\star\phi(x)=\phi(x/p)$ si $x\notin\Z_p$ et
$p\star\phi(x)=0$ si $x\in\Z_p$. Pour cela, on utilise la suite exacte de cohomologie
associ\'ee \`a la suite exacte
$$0\to {\cal C}({\bf P}^1\moins\Z_p, \O_L/p^n)\to {\cal C}(\Q_p\moins \Z_p, \O_L/p^n)\to V_n\to 0.$$
On a $H^0(p^{k\N},{\cal C}({\bf P}^1\moins\Z_p, \O_L/p^n))=(\O_L/p^n){\bf 1}_{{\bf P}^1\moins\Z_p}$,
et $H^0(p^{k\N},{\cal C}(\Q_p\moins \Z_p, \O_L/p^n))$ s'identifie au sous-espace
de ${\cal C}(\Q_p^\dual, \O_L/p^n)$ des fonctions v\'erifiant $\phi(p^kx)=\phi(x)$ pour tout $x$;
on a donc 
$$\varinjlim\nolimits_kH^0(p^{k\N},{\cal C}(\Q_p\moins \Z_p, \O_L/p^n))\cong
{\cal C}(\widehat{\Q_p^\dual},\O_L/p^n)$$

Passons au calcul de $H^1(p^{k\N},{\cal C}({\bf P}^1\moins\Z_p, \O_L/p^n))$.
On a une d\'ecomposition $p^{k\N}$-\'equivariante
${\cal C}({\bf P}^1\moins\Z_p, \O_L/p^n)=(\O_L/p^n){\bf 1}_{{\bf P}^1\moins\Z_p}\oplus X_n$,
o\`u $X_n$ est le sous-espace des $\phi$ avec $\phi(\infty)=0$.
On a $H^1(p^{k\N},X_n)=0$ car, si $\phi\in X_n$, l'\'equation $p^k\star\phi'-\phi'=\phi$ admet comme solution
$\phi'$, avec $-\phi'(x)=\phi(x)+\phi(p^{-k}x)+\phi(p^{-2k}x)+\cdots$ (la s\'erie est en fait une somme finie).
Par contre $H^1(p^{k\N}, (\O_L/p^n){\bf 1}_{{\bf P}^1\moins\Z_p})$ est de rang~$1$ sur $\O_L/p^n$,
engendr\'e par le bord de $\phi_k:=1_{\Q_p\moins\Z_p}v_p(x/p^k)$ 
(qui ne v\'erifie pas $p^k\star\phi_k=\phi_k$
mais v\'erifie $p^{p^nk}\star\phi_k=\phi_k$, et donc sa classe de cohomologie meurt en restriction \`a 
$p^{p^nk\N}$).  Il s'ensuit que
$$\varinjlim\nolimits_kH^0(p^{k\N},V_n)\cong \frac{{\cal C}(\widehat{\Q_p^\dual},\O_L/p^n)}{\{{\rm Ctes}\}}$$
Le r\'esultat s'en d\'eduit en passant \`a la limite sur $n$.
\end{proof}

\Subsubsection{Le r\'esultat principal}
\begin{theo}\label{image}
L'application $H^1({H_{\Q_p}}, \breve{\O}_L)\rightarrow 
\frac{{\cal C}(\Q_p, \O_L)}{{\cal C}({\bf P}^1, \O_L)}\otimes\chi^{-1}$ 
du diag.~{\rm(\ref{diagfond})}
se factorise \`a travers $\frac{{\cal C}(\widehat{\Q^{\dual}_p}, \O_L)}{\{{\rm Ctes}\}}$ et induit un 
diagramme ${\Q_p^{\dual}}$-\'equivariant dans lequel les fl\`eches verticales sont des isomorphismes
$$
\xymatrix@R=4mm@C=5mm{0 \ar[r] & \widetilde{\O}_{\cal E}^{-} \ar[r]\ar[d]^-{\wr}
& (\O_L\point\tA^-)^{H_{\Q_p}}\ar[r] \ar[d]^-{\wr}
& H^1({H_{\Q_p}}, \breve{\O}_L)\ar[r]\ar[d]^-{\wr} & 0 \\
0\ar[r]& {\cal C}(\Q_p, \O_L)_0\otimes\chi^{-1}\ar[r] 
& \frac{{\cal C}^{\rm pp}(\Q_p, \O_L)}{\{{\rm Ctes}\}}\otimes\chi^{-1}\ar[r] 
& \frac{{\cal C}(\widehat{\Q_p^\dual}, \O_L)}{\{{\rm Ctes}\}}\otimes\chi^{-1}\ar[r] & 0
}
$$
\end{theo}
\begin{proof}
Compte-tenu du lemme~\ref{ppp}, il suffit de prouver le premier \'enonc\'e et la bijectivit\'e
de la fl\`eche de droite.

Comme l'application $H^1({H_{\Q_p}}, \breve{\O}_L(1))\rightarrow \frac{{\cal C}(\Q_p, \O_L)}{{\cal C}({\bf P}^1, \O_L)}$ 
est $\Q_p^{\dual}$-\'equivariante, le premier \'enonc\'e r\'esulte
 de ce que l'action de $\Q_p^{\dual}$ 
sur $H^1({H_{\Q_p}}, \breve{\O}_L)\subset \breve{\O}_L\widehat{\otimes} H^1(H_{\Q_p}', \Z_p)$ 
est la restriction d'une action continue de $\widehat{\Q_p^\dual}=G_{\Q_p}/H_{\Q_p}'$
et du lemme~\ref{image2}.

L'injectivit\'e de la fl\`eche qui s'en d\'eduit r\'esulte de celle de la fl\`eche initiale.
 Il ne reste que la surjectivit\'e \`a prouver.
 Si $K$ est une extension finie de $\Q_p$, ab\'elienne, notons $\Gamma_K^{'}={\rm Gal}(\Q_p^{\rm ab}/K)$.
 Si $K$ contient $\bmu_{2p}$, alors $\Gamma_K^{'}$ est d'indice fini dans $(1+2p\Z_p)\times \varphi^{\widehat{\Z}}$, et donc est isomorphe \`a $\Z_p\times \widehat{\Z}$. 

Pour prouver la surjectivit\'e, il suffit de la prouver modulo $p$ et, pour cela, il suffit de v\'erifier que les invariants par $\Gamma_K^{'}$ ont la m\^eme dimension pour $K$ assez grand.
 Maintenant, on a $H^1({H_{\Q_p}}, \overline{\bf F}_p(1))=(\overline{\bf F}_p\otimes H^1(H_{\Q_p}', {\bf F}_p(1)))^{{H_{\Q_p}}/H_{\Q_p}'}$ car les $H^i({H_{\Q_p}}/H_{\Q_p}', \overline{\bf F}_p(1))$, pour $i=1,2$, sont nuls vu que ${H_{\Q_p}}/H_{\Q_p}'={\rm Gal}(\overline{\bf F}_p/{\bf F}_p)$ et $\overline{\bf F}_p(1)=\overline{\bf F}_p$ en tant que ${H_{\Q_p}}$-module.
 Or
\[\dim_{{\bf F}_p}(\overline{\bf F}_p\otimes V)^{{\rm Gal}(\overline{\bf F}_p/{\bf F}_p)}=\dim_{{\bf F}_p}V, \]
si $V$ est une ${\bf F}_p$-repr\'esentation lisse de ${\rm Gal}(\overline{\bf F}_p/{\bf F}_p)$.
 On est donc ramen\'e \`a calculer 
$\dim_{{\bf F}_p}H^1(H_{\Q_p}', {\bf F}_p(1))^{\Gamma_K^{'}}$ pour $K$ assez grand 
(et ${\bf F}_p(1)={\bf F}_p$ comme $G_K$-module).

 On a une suite exacte
\[0\rightarrow H^1(\Gamma^{'}_K, {\bf F}_p(1))\rightarrow H^1(G_K, {\bf F}_p(1))\rightarrow H^1(H_{\Q_p}', {\bf F}_p(1))^{\Gamma_K^{'}}\rightarrow H^2(\Gamma_K^{'}, {\bf F}_p(1)).\]
Comme $\Gamma_K^{'}=\Z_p^2\times\Gamma_K^{''}$, 
o\`u $\Gamma_K^{''}$ est un groupe profini de cardinal premier \`a $p$, et comme ${\bf F}_p(1)={\bf F}_p$, on a 
$H^i(\Gamma'_K, {\bf F}_p(1))\cong H^i(\Z_p^2,H^0(\Gamma_K^{''},{\bf F}_p))=
H^i(\Z_p^2,{\bf F}_p)$, et donc
\[\dim_{{\bf F}_p}H^1(\Gamma'_K, {\bf F}_p(1))=2 \text{ et } \dim_{{\bf F}_p}H^2(\Gamma'_K, {\bf F}_p(1))=1.\]
Maintenant $H^i(G_K,{\bf F}_p(1))$ est de dimension $1$, si $i=0,2$,
et donc la formule d'Euler-Poincar\'e plus la dualit\'e locale, donnent 
\[\dim_{{\bf F}_p}H^1(G_K, {\bf F}_p(1))=[K:\Q_p]+2.\]
 On en d\'eduit que 
\[\dim_{{\bf F}_p}H^1(H_{\Q_p}', {\bf F}_p(1))^{\Gamma'_K}\leq [K:\Q_p]+1.\]
Par ailleurs, on a 
\[\dim_{{\bf F}_p}H^1(H_{\Q_p}', {\bf F}_p(1))^{\Gamma'_K}\geq \dim_{\Q_p}H^1(H_{\Q_p}', \Q_p(1))^{\Gamma'_K}\]
et $H^1(H_{\Q_p}', \Q_p(1))^{\Gamma'_K}$ vit dans une suite identique \`a la suite ci-dessus dans laquelle on remplace ${\bf F}_p(1)$ par $\Q_p(1)$.
 La diff\'erence est que l'on a $H^i(\Gamma'_K, \Q_p(1))=0$, si $i=1,2$ (car $H^i(\Gamma_K, \Q_p(1))=0$ si $i=0,1$, o\`u $\Gamma_K=\Gamma'_K\cap \Z_p^{\dual}$).
 La formule d'Euler-Poincar\'e (ou la th\'eorie de Kummer) donne $\dim_{\Q_p}H^1(G_K, \Q_p(1))=[K:\Q_p]+1$ et donc $\dim_{\Q_p}H^1(H_{\Q_p}', \Q_p(1))^{\Gamma'_K}=[K:\Q_p]+1$ et, enfin,
\[\dim_{{\bf F}_p}H^1(H_{\Q_p}', {\bf F}_p(1))^{\Gamma'_K}=[K:\Q_p]+1.\]
Pour calculer $({\cal C}(\widehat{\Q_p^{\dual}}, {\bf F}_p)/{\bf F}_p)^{\Gamma'_K}$, on utilise la suite exacte
\[0\rightarrow {\bf F}_p\rightarrow {\cal C}(\widehat{\Q_p^{\dual}}, {\bf F}_p)^{\Gamma'_K}\rightarrow ({\cal C}(\widehat{\Q_p^{\dual}}, {\bf F}_p)/{\bf F}_p)^{\Gamma'_K}\rightarrow H^1(\Gamma'_K, {\bf F}_p)\rightarrow H^1(\Gamma'_K, {\cal C}(\widehat{\Q_p^{\dual}}, {\bf F}_p)).\]
Or ${\cal C}(\widehat{\Q_p^{\dual}}, {\bf F}_p)\cong {\cal C}(\Gamma'_K, {\bf F}_p)^{[\widehat{\Q_p^{\dual}}: \Gamma'_K]}$ en tant que $\Gamma'_K$-module et le lemme de Shapiro nous donne ${\cal C}(\Gamma'_K,{\bf F}_p)^{\Gamma'_K}={\bf F}_p$ et $H^1(\Gamma'_K, {\cal C}(\Gamma'_K, {\bf F}_p))=0$.
 Comme $\dim_{{\bf F}_p}H^1(\Gamma'_K, {\bf F}_p)=2$, on obtient
\[ \dim_{{\bf F}_p}({\cal C}(\widehat{\Q_p^{\dual}}, 
{\bf F}_p)/{\bf F}_p)^{\Gamma'_K}=[\widehat{\Q_p^{\dual}}: \Gamma'_K]+2-1=[K:\Q_p]+1.\qedhere\]
\end{proof}

\Subsection{Espaces propres pour l'action de $\Q_p^\dual$}
\begin{lemm}\label{ext12}
Soit $\delta:\Q_p^\dual\to\O_L^\dual$ un caract\`ere continu. Alors
$$H^0(\Q_p^\dual,\tfrac{{\cal C}(\widehat{\Q_p^\dual},\O_L)}{\{{\rm Ctes}\}}\otimes e_\delta)=
\begin{cases}\O_L\,\delta\otimes e_\delta&{\text{si $\delta\neq 1$,}}\\ 
{\rm Hom}(\Q_p^\dual,\O_L)\otimes e_\delta&{\text{si $\delta= 1$}}.\end{cases}$$
\end{lemm}
\begin{proof}
Le r\'esultat est imm\'ediat si $\delta\neq 1$. Si $\delta=1$, on cherche les fonctions $\phi$ telles
que $x\mapsto\phi(ax)-\phi(x)$ est une constante pour tout $a$ (modulo les $\phi$ constantes).  Quitte
\`a retirer une fonction constante, on peut supposer que $\phi(1)=0$, 
et on a alors $\phi(ax)=\phi(x)+\phi(a)$ pour tous $a\in\Q_p^\dual$ et $x\in \widehat{\Q_p^\dual}$.
Le r\'esultat s'en d\'eduit.
\end{proof}

Nous allons d\'ecrire la compos\'ee $\iota_\delta$ de la
suite d'applications (cf.~(ii) du lemme~\ref{ext10} et diag.\,(\ref{diagfond}))
 \[H^1(G_{\Q_p}, {\bf V}(\delta))\to H^1(H_{\Q_p}', {\bf V}(\delta))^{\widehat{\Q_p^{\dual}}}
\to \big(H^1(H_{\Q_p},\breve{L})\otimes e_{\delta}\big)^{\Q_p^\dual}
\rightarrow \big(\tfrac{{\cal C}(\Q_p, L)}{{\cal C}({\bf P}^{1}, L)}\otimes e_{\chi^{-1}\delta}\big)^{\Q_p^\dual}\]
Pour \'enoncer le r\'esultat, rappelons
que, si $K$ est une extension finie de $\Q_p$, on dispose d'une application de Kummer
$${\rm Kum}_K:K^\dual\to H^1(G_K,\Q_p(1))$$
 dont l'image est un $\Z_p$-r\'eseau de $H^1(G_K,\Q_p(1))$.
Par ailleurs ${\rm Hom}(\Q_p^\dual,L)$ est de dimension~$2$, et poss\`ede une base naturelle
$(v_p,\tau)$, o\`u 
$$\tau=p^{-c(p)}\log,\quad{\text{avec $c(p)=1$ (resp.~$2$) si $p\neq 2$ (resp.~$p=2$)}}$$
et $\log$ est le logarithme d'Iwasawa (d\'efini par $\log p=0$), de telle sorte
que  $(v_p,\tau)$ est une base de ${\rm Hom}(\Q_p^\dual,\O_L)$ sur $\O_L$.
Soit $F=\Q_p(\bmu_{p^{c(p)}})$.
 \begin{prop}\label{ext13}
 {\rm (i)} Si $\delta\neq 1,\chi$, 
l'image de $\iota_\delta$ est $L\cdot\chi^{-1}\delta\otimes e_{\chi^{-1}\delta}$.

{\rm (ii)} Si $\delta=\chi$, l'image de $\iota_\delta$ est ${\rm Hom}(\Q_p^\dual,L)\otimes{e_1}$.
Plus pr\'ecis\'ement, si $\alpha\in\Q_p^\dual$,
\begin{align*}
\iota_\chi({\rm Kum}_{\Q_p}(\alpha))
&=[F:\Q_p](v_p(\alpha)\cdot \tau-\tau(\alpha)\cdot v_p)\otimes{e_1}\\
&=(1-\tfrac{1}{p})(v_p(\alpha)\cdot \log-\log(\alpha)\cdot v_p)\otimes{e_1}
\end{align*}

{\rm (iii)} Si $\delta=1$, alors $\iota_\delta$ est identiquement nulle.
 \end{prop}
\begin{proof}Si $\delta\neq\chi$, cela r\'esulte des lemmes~\ref{ext12} et~\ref{ext10}(ii).

Si $\delta=\chi$, 
on utilise la th\'eorie des $(\varphi, \Gamma)$-modules pour d\'ecrire $H^1(G_{\Q_p}, \Q_p(1))$, et le
r\'esultat est une cons\'equence de la prop.\,\ref{aj2} et du lemme~\ref{aj3}.
\end{proof}

\subsubsection{$(\varphi,\Gamma)$-modules et application de Kummer}
Soit $a$ le g\'en\'erateur topologique de $1+p^{c(p)}\Z_p$:
$$a:=\exp(p^{c(p)})$$ 
Alors $\Gamma=\sigma_a^{\Z_p}\times\Delta$, avec $\Delta\cong {\rm Gal}(F/\Q_p)$.

Si $K=\Q_p,F$,
la th\'eorie locale du corps de classes fournit des isomorphismes
\begin{align*}
{\rm cl}_K: {\rm Hom}(K^*, \Q_p)\cong H^1(G_{K}, \Q_p),
\quad {\rm Tr}_{K}: H^2(G_{K}, \Q_p(1))\cong \Q_p
\end{align*}
et, si $\ell\in {\rm Hom}(K^*, \Q_p)$ et $\alpha\in K^\dual$, on a:
\begin{equation}\label{pair}
{\rm Tr}_{K}({\rm cl}_K(\ell)\cup {\rm Kum}_K(\alpha))= \ell(\alpha)
\end{equation}
Par ailleurs, si $x\in H^1(G_{\Q_p},\Q_p)$ et $y\in H^1(G_{\Q_p},\Q_p(1))$,
\begin{align}\label{trace}
{\rm Tr}_F({\rm res}^F_{\Q_p} x\cup {\rm res}^F_{\Q_p} y)&=[F:\Q_p]\,{\rm Tr}_{\Q_p}(x\cup y)
\end{align}

Si $V$ est une $L$-repr\'esentation de $G_{\Q_p}$, on a $H^i(G_{\Q_p},V)=H^i(G_F,V)^\Delta$, et 
si $D={\bf D}(V)$,
les\footnote{Passer par $F$ permet de traiter sur le m\^eme plan les cas $p\neq 2$ o\`u $\Gamma$ est procyclique
et $p=2$ o\`u il ne l'est pas.}
 $H^i(G_F,V)$ sont naturellement isomorphes aux groupes de cohomogie du complexe
 \[C^\bullet(D):=\big[\xymatrix@C=3.5cm{D\ar[r]^-{x\mapsto ( (\varphi-1)x,(\sigma_a-1)\point x)} 
&D\oplus D\ar[r]^-{(u,v)\mapsto (1-\sigma_a)u+(\varphi-1)v} & D }\big] \]
On note $H^i(D)$ (resp.~$Z^i(D)$), pour $i=0,1,2$, les groupes de cohomologie (resp.~de cocycles)
du complexe $C^\bullet(D)$.
On note $h^i_V:H^i(D)\overset{\sim}{\to}H^i(G_F,V)$ l'isomorphisme ci-dessus et aussi
$h^i_V:Z^i(D)\to H^i(G_F,V)$ l'application qui s'en d\'eduit.

Si $(u,v)\in Z^1(D)$, et si $\tilde{c}\in \widetilde{\bf A}\otimes_{\A_{\Q_p}}D$ 
v\'erifie $(\varphi-1)\tilde{c}=u$, le cocycle associ\'e \`a $(u,v)$ est\footnote{
Notons $D_{u,v}$ l'extension de $\oe$ par $D$ associ\'ee \`a $(u,v)$: il existe un rel\`evement
$e\in D_{u,v}$ de $1\in \oe$ avec $(\varphi-1)e=u$ et $(\sigma_a-1)e=v$. Si $V_{u,v}$ est
l'extension de $\O_L$ par $V$ correspondante, on a $V_{u,v}=(\tA\otimes D_{u,v})^{\varphi=1}$
et $e-\tilde c$ est un rel\`evement de $1\in\O_L$ dans $V_{u,v}$; le cocycle correspondant est donc
$\sigma\mapsto (\sigma-1)(e-\tilde c)=\tfrac{\sigma-1}{\sigma_a-1}(\sigma_a-1)e-(\sigma-1)\tilde{c}=
c^{u,v}_\sigma$.}
\begin{equation}\label{coc}
\sigma\mapsto c^{u,v}_\sigma:=\tfrac{\sigma-1}{\sigma_a-1}v-(\sigma-1)\tilde{c}
\in (\widetilde{\bf A}\otimes_{\A_{\Q_p}}D)^{\varphi=1}=V
\end{equation}

Si $h^i_V(z)\in H^i(G_F,V)^\Delta=H^i(G_{\Q_p},V)$, on note $h^{i,\Delta}_V(z)$ l'image
dans $H^i(G_{\Q_p},V)$ (i.e.~${\rm res}_{\Q_p}^Fh^{i,\Delta}_V(z)=h^i_V(z)$).

\begin{exem}\label{aj1}
Soit $\B_{\Q_p}=\A_{\Q_p}[\frac{1}{p}]$.

(i) Si $D=\B_{\Q_p}$, alors $H^1(D)^\Delta$ admet $(1,0)$ et $(0,1)$ comme base.

(ii) Si $D=\B_{\Q_p}\otimes\chi$, 
alors $H^1(D)^\Delta$ admet une base naturelle $(u_1,v_1)$, $(u_2,v_2)$ o\`u
$u_1=\frac{1}{\pi}+\frac{1}{2}$, $v_2=\frac{1}{\pi}$, $v_1\in\A_{\Q_p}^+$ et
$\psi(u_2)=0$.
\end{exem}
(Tout ceci est parfaitement classique: pour construire $(u_1,v_1)$, un petit calcul montre que
$(u\sigma_u-1)\big(\frac{1}{\pi}+\frac{1}{2}\big)\in\pi\A_{\Q_p}^+$ pour tout $u\in\Z_p^\dual$,
et on utilise la bijectivit\'e de $\varphi-1$ sur $\pi\A_{\Q_p}^+$;
pour $(u_2,v_2)$, on utilise la bijectivit\'e de $a\sigma_a-1$ sur $\A_{\Q_p}^{\psi=0}$
et le fait que $\psi(\frac{1}{\pi})=\frac{1}{\pi}$ et donc 
$(\varphi-1)\frac{1}{\pi}\in \A_{\Q_p}^{\psi=0}$.)

\begin{prop}\label{aj2}
{\rm (i)} On a $h^{1,\Delta}_{\Q_p}(1,0)=v_p$ et $h^{1,\Delta}_{\Q_p}(0,1)=\tau$.

{\rm (ii)} On a $h^{1,\Delta}_{\Q_p(1)}(v_1,u_1)=\frac{-1}{[F:\Q_p]}{\rm Kum}_{\Q_p}(a)$
et $h^{1,\Delta}_{\Q_p(1)}(v_2,u_2)=\frac{1}{[F:\Q_p]}{\rm Kum}_{\Q_p}(p)$.
\end{prop}
\begin{proof}
Le (i) se d\'eduit facilement de la formule~(\ref{coc}) (cf.~\cite[\no3.2.3]{CGN}).
Pour d\'emontrer le (ii), le plus simple est d'utiliser la dualit\'e: 
notons $\check D={\rm Hom}_{\A_{\Q_p}}(D,\A_{\Q_p})\otimes\chi$ le dual de Tate de $D$
(la repr\'esentation associ\'ee est ${\rm Hom}(V,\Q_p(1))$).
Si $(\check u,\check v)\in Z^1(\check D)$ et $(u,v)\in Z^1(D)$, alors (cf.~\cite{Benois} ou~\cite[\no3.2.4]{CGN}):
$$h^1_{\check{V}}(\check u,\check v)\cup h^1_V(u,v)=h^2_{\Q_p(1)}(\langle\check u,\sigma_a(v)\rangle-
\langle \check v,\varphi(u)\rangle)$$
Par ailleurs, si
$z\in \A_{\Q_p}\otimes\chi$, alors $h^2_{\Q_p(1)}(z)\in H^2(G_F,\Q_p(1))$ 
et ${\rm Tr}_F(h^2_{\Q_p(1)}(z))=\reso\big(z\frac{d\pi}{1+\pi}\big)$ 
o\`u $\reso\big(\sum a_k\pi^k\,d\pi\big)=a_{-1}$.

On a $\reso\big(v_1\frac{d\pi}{1+\pi}\big)=0$ car $v_1\in\A_{\Q_p}^+$ et
$\reso\big(u_2\frac{d\pi}{1+\pi}\big)=0$ car $\psi(u_2)=0$. On en d\'eduit les formules
\begin{align*}
{\rm Tr}_F(h^1_{\Q_p(1)}(u_1,v_1)\cup h^1_{\Q_p}(1,0))=0,\quad
{\rm Tr}_F(h^1_{\Q_p(1)}(u_1,v_1)\cup h^1_{\Q_p}(0,1))=1\\
{\rm Tr}_F(h^1_{\Q_p(1)}(u_2,v_2)\cup h^1_{\Q_p}(1,0))=-1,\quad
{\rm Tr}_F(h^1_{\Q_p(1)}(u_2,v_2)\cup h^1_{\Q_p}(0,1))=0
\end{align*}
On en d\'eduit le r\'esultat en utilisant le (i) et la formule (\ref{pair}).
(Le $\frac{1}{[F:\Q_p]}$ provient de la relation~(\ref{trace}).)
\end{proof}

\subsubsection{$(\varphi,\Gamma)$-modules et d\'etermination de $\iota_\chi$}
Soient $D=\A_{\Q_p}\otimes e_\chi$ et $(u,v)\in Z^1(D)$ d'image dans $H^1(D)$ invariante par $\Delta$. 
La restriction
\`a ${H_{\Q_p}}$ de $\sigma\mapsto c^{u,v}_\sigma$ (cf.~formule~(\ref{coc}))
est $\sigma\mapsto -(\sigma-1)\tilde{c}$, et $\tilde c\in(\tA^-)^{H_{\Q_p}}\point e_\chi$ (puisque
$c^{u,v}_\sigma\in\Z_p\point e_\chi$). 

Pour calculer
$\iota_\chi(h^{1,\Delta}_{\Q_p(1)}(u,v))$, il s'agit de calculer la fonction $\phi_{\tilde c}$ associ\'ee
\`a $\tilde c$ par la prop.\,\ref{ext8} (plus pr\'ecis\'ement, il r\'esulte du lemme~\ref{ext12} qu'il
existe $\ell\in {\rm Hom}(\Q_p^\dual,\O_L)$ et $C\in\O_L$ tels que $\phi_{\tilde c}-\ell=C$ sur $p^{-N}\Z_p$ si $N$
est assez grand, et alors $\iota_\chi(h^{1,\Delta}_{\Q_p(1)}(u,v))=\ell$).
Pour ce faire, il faut revenir \`a la preuve de la prop.\,\ref{ext8} dont
nous reprenons certaines des notations ci-dessous.

 Si $b\in \Q_p$, il existe $\tilde{c}(b)\in \widetilde{\bf A}^{++}$ tel que 
$(\sigma-1)\tilde{c}(b)=([\epsilon^b]-1)c^{u,v}_\sigma$, pour tout $\sigma\in {H_{\Q_p}}$, 
et donc $v_b:=([\epsilon^b]-1)\tilde{c}+\tilde{c}(b)\in \widetilde{\bf A}_{\Q_p}$, 
et $v_b$ modulo $\widetilde{\bf A}_{\Q_p}^{++}$ ne d\'epend pas du choix de $\tilde{c}(b)$. 
Notons que l'on peut prendre $\tilde{c}(pb)=\varphi(\tilde{c}(b))$ car $\varphi(c^{u,v}_\sigma)=c^{u,v}_\sigma$
et $\varphi([\epsilon^b]-1)=[\epsilon^{pb}]-1$.
Comme $\varphi(\tilde{c})=\tilde{c}+u$, on a aussi 
$$\varphi(v_b)=v_{pb}+([\epsilon^{pb}]-1)u$$

Soient $\phi_b$ et $\phi_{pb}\in {\cal C}(\Q_p, L)_0$ les fonctions associ\'ees aux images de $v_b, v_{pb}$ 
dans $\widetilde{\bf A}^{-}_{\Q_p}$ et soit $\phi_u\in {\cal C}(\Z_p,\Z_p)$ 
la fonction associ\'ee \`a l'image de $u$ dans ${\bf A}_{\Q_p}^{-}$.
 Alors la relation ci-dessus se traduit par 
\begin{equation}\label{ide}
\phi_b(p^{-1}x)=\phi_{pb}(x)+\phi_u(x-pb)-\phi_u(x), \quad \text{ pour tout } x\in \Q_p.
\end{equation}
Par construction,
$\phi_{\tilde c}\in {\cal C}(\Q_p,\Z_p)$ 
satisfait la relation $\phi_b(x)=\phi_{\tilde c}(x-b)-\phi_{\tilde c}(x)$ pour tous $b,x\in \Q_p$.
 Si on reporte cette relation dans l'identit\'e~(\ref{ide}), 
on obtient que $x\mapsto \phi_{\tilde c}(p^{-1}x)-\phi_{\tilde c}(x)-\phi_u(x)$ 
est invariant par $x\mapsto x-pb$, pour tout $b$, et donc est une constante, ce qui
fournit un moyen de calculer $\phi_{\tilde c}$ si on connait $\phi_u$.

\begin{lemm} \label{aj3}
On a les relations
$$\iota_{\chi}\big(h^{1,\Delta}_{\Q_p(1)}(u_1,v_1)\big)=v_p
\quad{\rm et}\quad
\iota_{\chi}\big(h^{1,\Delta}_{\Q_p(1)}(u_2,v_2)\big)=\tau$$
\end{lemm}
\begin{proof}

$\bullet$ Comme $u_1=\frac{1}{\pi}+\frac{1}{2}$, on a $\phi_{u_1}={\bf 1}_{\Z_p}$, et il existe $C$ tel que
$\phi_{\tilde c}(p^{-1}x)-\phi_{\tilde c}(x)=C+{\bf 1}_{\Z_p}(x)$ pour tout $x\in\Q_p^\dual$.
Comme $\phi_{\tilde c}$ se prolonge par continuit\'e en $0$, on en d\'eduit $C=-1$, 
et $\phi_{\tilde c}=v_p+C'$ sur $\Q_p\moins \Z_p$. On en tire la premi\`ere relation.

$\bullet$ Comme $v_2=\frac{1}{\pi}$, on a $\phi_{v_2}={\bf 1}_{\Z_p}$.
 Donc $\phi_{(\varphi-1)v_2}=-{\bf 1}_{\Z_p^{\dual}}$.
 La relation $(a\sigma_a-1)u_2=(\varphi-1)v_2$ se traduit par\footnote{Le $a$ disparait gr\^ace au
twist $\chi^{-1}$ du dictionnaire d'analyse $p$-adique.}
$\phi_{u_2}(a^{-1}x)-\phi_{u_2}(x)=\phi_{(\varphi-1)v_2}=-{\bf 1}_{\Z_p^{\dual}}$.

 Par ailleurs, comme $\psi(u_2)=0$, la fonction $\phi_{u_2}(x)$ est \`a support dans $\Z_p^{\dual}$.
 Donc, $\phi_{u_2}={\bf 1}_{\Z_p^\dual}\cdot (\tau+C)$, avec $C\in \Z_p$.
Les solutions de $\phi_{\tilde c}(p^{-1}x)-\phi_{\tilde c}(x)=\phi_{u_2}(x)$ admettant un prolongement
continu en $0$ sont de la forme ${\bf 1}_{\Q_p\moins p\Z_p}\tau+ C\,{\bf 1}_{\Z_p} + C'$.
On en tire la seconde relation.
\end{proof}

\section{Application aux repr\'esentations de $\GG$}\label{chap5}
\Subsection{Les s\'eries principales de $\GG$}
\subsubsection{L'espace des param\`etres}
Soit $\widehat{\cal T}^0$ l'espace des caract\`eres unitaires de $\Q_p^\dual$: c'est le produit de l'espace
des caract\`eres de $\Z_p^\dual$ (une r\'eunion de $p-1$ boules ouvertes -- $2$ si $p=2$) 
par celui des caract\`eres unitaires de $p^\Z$ (un cercle).

Soit ${\cal S}^0$ l'\'eclat\'ee de $\widehat{\cal T}^0\times\widehat{\cal T}^0$ en la sous-vari\'et\'e
des $(\delta_1,\delta_2)$ avec $\delta_1=\chi\delta_2$. On a donc un morphisme
surjectif $\pi:{\cal S}^0\to \widehat{\cal T}^0\times\widehat{\cal T}^0$ dont la fibre
en $(\delta_1,\delta_2)$ est un point sauf si $\delta_1=\chi\delta_2$ o\`u cette fibre
est isomorphe \`a $\piqp$. On identifie ce $\piqp$ au projectivis\'e de l'espace des logarithmes
de $\Q_p^\dual$; et si ${\cal L}$ est une droite de l'espace des logarithmes, on note
$(\delta_1,\delta_2,{\cal L})$ le point de ${\cal S}^0$ correspondant.

\subsubsection{La repr\'esentation $V_z$ de $G_{\Q_p}$}
On associe \`a $z\in{\cal S}^0(L)$ une repr\'esentation $V_z$ de $G_{\Q_p}$:

$\bullet$ Si $\delta_1\neq\delta_2,\chi\delta_2$, on prend pour $V_z$ l'unique extension
non triviale de $\delta_2$ par $\delta_1$.

$\bullet$ Si $\delta_1=\chi\delta_2$, on prend pour $V_z$ l'extension de $\delta_2$ par $\chi\delta_2$
dont la classe dans ${\rm Ext}^1(\delta_2,\chi\delta_2)\equiv H^1(G_{\Q_p},L(\chi))$
est l'orthogonal de ${\cal L}$ pour l'accouplement
$H^1(G_{\Q_p},L(\chi))\times H^1(G_{\Q_p},L)\to H^2(G_{\Q_p},L(\chi))=L$
(combin\'ee avec l'identification $H^1(G_{\Q_p},L)\cong {\rm Hom}(\Q_p^\dual,L)$ de
la th\'eorie locale du corps de classe).

$\bullet$ Si $\delta_1=\delta_2$, on pose $V_z=\delta_1\oplus\delta_1$.

\begin{rema}\label{pathos1}
Ce qui se passe au voisinage de $\{\delta_1=\delta_2\}$ est peu clair.
Si $\ell\in{\rm Hom}(\Q_p^\dual,\O_L)$, consid\'erons la famille de repr\'esentations
de $G_{\Q_p}$ de matrice~$\matrice{\exp(u\ell)}{\frac{\exp(u\ell)-1}{u}}{0}{1}$: elle
est scind\'ee pour $u\neq 0$ et, pour $u=0$, on obtient une extension non triviale de $1$ par $1$
dont la classe est donn\'ee par $\ell$.  Il serait peut-\^etre plus naturel de retirer
le lieu $\{\delta_1=\delta_2\}$ de ${\cal S}^0$. 
\end{rema}

\subsubsection{La repr\'esentation $\Pi_z$ de $\GG$}
On associe \`a $z\in{\cal S}^0(L)$ une repr\'esentation $\Pi_z$ de $\GG$ dont la restriction
\`a $\PP$ vit dans une suite exacte $\PP$-\'equivariante, non scind\'ee,
$$0\to \Pi_{z,0}\to\Pi_z\to J_z\to 0$$ o\`u $J_z$ est le module de Jacquet de $\Pi_z$ et,
si $\pi(z)=(\delta_1,\delta_2)$, alors 
$$\Pi_{z,0}\cong{\cal C}(\Q_p,L)_0\otimes\chi^{-1}\delta_1
\quad{\rm et}\quad
J_z\cong \delta_2$$
 La recette est la suivante:

$\bullet$ Si $\delta_1\neq\chi\delta_2$, 
on pose $\Pi_z:={\rm Ind}_\BB^\GG(\delta_2\otimes \delta_1\chi^{-1})$.

$\bullet$ Si $\delta_1=\chi\delta_2$, alors ${\rm Ind}_\BB^\GG(\delta_2\otimes \delta_1\chi^{-1})$
contient une sous-repr\'esentation
de dimension~$1$, isomorphe \`a $\delta_2$ (identifi\'e au caract\`ere
$\delta_2\circ\det$ de $\GG$) et le quotient est ${\rm St}\otimes\delta_2$,
o\`u ${\rm St}$ est la steinberg continue. Si $z=(\delta_1,\delta_2,{\cal L})$, o\`u
${\cal L}$ est une droite de ${\rm Hom}(\Q_p^\dual,L)$, on construit $\Pi_z$
comme une extension non triviale de $\delta_2$ par ${\rm St}\otimes\delta_2$:
on a $\Pi_{(\chi\delta_2,\delta_2,{\cal L})}=\Pi_{(\chi,1,{\cal L})}\otimes\delta_2$
et pour contruire $\Pi_{(\chi,1,{\cal L})}$,
on choisit une base $\ell$ de ${\cal L}$
et on d\'efinit $\Pi_{(\chi,1,{\cal L})}$ comme un sous-quotient
de l'induite de $\BB$ \`a $\GG$ de la repr\'esentation $W_\ell$ de dimension $2$,
o\`u $\matrice{a}{b}{0}{d}$ agit par $\matrice{1}{\ell(a/d)}{0}{1}$.
On a une suite exacte $\GG$-\'equivariante
$$0\to {\rm Ind}_\BB^\GG(1\otimes1)\to {\rm Ind}_\BB^\GG W_\ell\to {\rm Ind}_\BB^\GG(1\otimes1)\to 0$$
et on d\'efinit $\Pi_{(\chi,1,{\cal L})}$ comme le quotient par ${\bf 1}\subset {\rm Ind}_\BB^\GG(1\otimes1)$ (de gauche)
de l'image inverse de ${\bf 1}\subset {\rm Ind}_\BB^\GG(1\otimes1)$ (de droite). On obtient donc
une extension de ${\bf 1}$ par $\big({\rm Ind}_\BB^\GG(1\otimes1)\big)/{\bf 1}={\rm St}$,
comme annonc\'e.

\vskip2mm
Par d\'efinition, ${\rm Ind}_\BB^\GG(\delta_2\otimes \delta_1\chi^{-1})$ est un espace
de fonctions sur $\GG$. En \'evaluant ces fonctions en $\matrice{0}{1}{-1}{x}$, on obtient
un plongement de $\Pi_z$ dans ${\cal C}(\Q_p,L)$, et m\^eme dans ${\cal C}^{\rm pp}(\Q_p,L)$
(si $\delta_1=\chi\delta_1$, cela ne fournit qu'un plongement de ${\rm St}\otimes\delta_2$; pour
obtenir un plongement de $\Pi_z$, on utilise les techniques de la preuve de la prop.\,\ref{ext8}).
On obtient alors un diagramme commutatif, $\PP$-\'equivariant:
$$
\xymatrix@R=4mm@C=5mm{0 \ar[r] & \Pi_{z,0} \ar[r]\ar[d]^-{\wr}
& \Pi_z\ar[r] \ar@{^{(}->}[d]
& J_z\ar[r]\ar@{^{(}->}[d] & 0 \\
0\ar[r]& {\cal C}(\Q_p, L)_0\otimes\chi^{-1}\delta_1\ar[r] 
& {\cal C}^{\rm pp}(\Q_p, L)\otimes\chi^{-1}\delta_1\ar[r] 
& {\cal C}(\widehat{\Q_p^\dual}, L)\otimes\chi^{-1}\delta_1\ar[r] & 0
}$$
et l'image de $J_z$ est $(L\cdot\chi\delta_2\delta_1^{-1})\otimes \chi^{-1}\delta_1$
si $\delta_1\neq\chi\delta_2$, et est ${\cal L}\otimes \delta_2$ si $\delta_1=\chi\delta_2$.

On a ${\bf V}(\Pi_z)={\bf V}(\delta_1)$, ce qui fournit un morceau de la repr\'esentation $V_z$
(et donc aussi une partie de la description de $z$) mais il semble difficile de r\'ecup\'erer
le reste en utilisant le foncteur ${\bf V}$.  Mais la construction de ${\bf V}$ n'utilise
que $\Pi_{z,0}$ et le th.\,\ref{catego} ci-dessous montre que $J_z$ contient le reste
de l'information.

On a un diagramme commutatif (on passe de la premi\`ere \`a la seconde ligne en inversant
le diagramme du th.\,\ref{image}; 
les deux fl\`eches verticales de droite sont surjectives mais pas injectives -- leurs noyaux
sont les fonctions constantes):
\[\xymatrix@R=4mm@C=5mm{
0\ar[r]& {\cal C}(\Q_p, L)_0\otimes\chi^{-1}\delta_1\ar[r]\ar[d]^-{\wr}
& {\cal C}^{\rm pp}(\Q_p, L)\otimes\chi^{-1}\delta_1\ar[r]\ar[d]          
& {\cal C}(\widehat{\Q_p^\dual}, L)\otimes\chi^{-1}\delta_1\ar[r]\ar[d] & 0\\
0\ar[r]& \we^-\otimes\delta_1\ar@{=}[d]\ar[r]&(\tB^-)^{H_{\Q_p}}\otimes\delta_1\ar@{=}[d]\ar[r]
&H^1(H_{\Q_p},\breve{L})\otimes\delta_1\ar@{=}[d]\ar[r]&0\\
0\ar[r]& \widetilde D({\bf V}(\delta_1))\ar[r]& (\tB^-\otimes {\bf V}(\delta_1))^{H_{\Q_p}}\ar[r]
&H^1(H_{\Q_p},\breve{L}(\delta_1))\ar[r]&0
}\]
En combinant ce diagramme avec le diagramme pr\'ec\'edent, on obtient
une injection $\PP$-\'equivariante $\Pi_z\hookrightarrow (\tB^-\otimes {\bf V}(\Pi_z))^{H_{\Q_p}}$
qui \'etait le point de d\'epart de cet article.  On obtient aussi une injection
$J_z\hookrightarrow H^1(H_{\Q_p},\breve{L}(\delta_1))$.  Rappelons que le groupe de droite
contient ${\rm Ext}({\bf V}(\delta),{\bf V}(\delta_1))$ pour tout $\delta$.
\begin{theo}\label{catego}
Si $\delta_1\neq \delta_2$,
l'image de $J_z$ dans $H^1(H_{\Q_p},\breve{L}(\delta_1))$ est une droite de
${\rm Ext}({\bf V}(\delta_2),{\bf V}(\delta_1))$ et l'extension de
${\bf V}(\delta_2)$ par ${\bf V}(\delta_1)$ correspondante est isomorphe \`a~$V_z$.
\end{theo}
\begin{proof}
Que l'image soit une droite est la traduction du fait que $J_z$ est de dimension~$1$ et n'est pas
contenue dans le noyau. Qu'elle soit incluse dans ${\rm Ext}({\bf V}(\delta_2),{\bf V}(\delta_1))$
r\'esulte de la $\PP$-\'equivariance et de la d\'etermination des sous-espaces propres
pour l'action de $\Q_p^\dual$ (cf.~(ii) du lemme~\ref{ext10}). 
Si $\delta_1\neq\chi\delta_2$, il y a une seule
extension possible, ce qui d\'emontre le r\'esultat dans ce cas.  Si $\delta_1=\chi\delta_2$,
que l'extension obtenue soit $V_z$ est une traduction du (ii) de la prop.\,\ref{ext13}.
\end{proof}

\begin{rema}\label{pathos2}
Si $\delta_1=\delta_2$, l'image de $J_z$ dans $H^1(H_{\Q_p},\breve{L}(\delta_1))$ fournit
une classe de $H^2(\widehat{\Q_p^\dual},L)$ (cf.~rem.\,\ref{ext11}), ce qui semble difficile
\`a interpr\'eter.  Ce n'est probablement pas sans lien avec la rem.\,\ref{pathos1}.
\end{rema}

\end{document}